\documentclass[final]{elsarticle}
\usepackage{lmodern}

\usepackage{hyperref}

\usepackage[normalem]{ulem}
\usepackage{subfigure}
\usepackage{graphicx}
\usepackage{amssymb,amsfonts}
\usepackage{amssymb,amsfonts,amsmath,latexsym,dsfont}
\usepackage[ruled]{algorithm2e}
\usepackage{amsthm}
\usepackage{mathtools}
\usepackage{wasysym}
\usepackage{wrapfig}

\usepackage{tikz,pgfplots}
\usetikzlibrary{plotmarks}

\usepackage{fullpage}
\usepackage{graphicx}
\usepackage{mathrsfs}
\usepackage{framed,multirow}
\usepackage{setspace}

\usepackage{booktabs}

\usepackage{soul}

\newdimen\iwidth
\newdimen\iheight

\newcommand{\mc}[3]{\multicolumn{#1}{#2}{#3}}

\graphicspath{{Fig/}{./}}

\renewcommand{\div}{\nabla\cdot}
\newcommand{\grad}{\nabla}

\newcommand{\bfe}{{\bf e}}

\newcommand{\bfp}{{\bf p}}
\newcommand{\bfx}{ {\bf x}}

\newcommand{\bfv}{ {\bf v}}

\newcommand{\bfu}{{\bf u}}
\newcommand{\bfq}{{\bf q}}

\newcommand{\bfd}{{\bf d}}

\newcommand{\bfw}{{\bf w}}

\newcommand{\bfgamma}{{\boldsymbol \gamma}}

\newcommand{\bfmu}{{\boldsymbol \mu}}
\newcommand{\bflambda}{{\boldsymbol \lambda}}
\newcommand{\bfrho}{{\boldsymbol \rho}}

\newtheorem{theorem}{Theorem}[section]

\newtheorem{definition}{Definition}[section]

\pgfplotsset{compat=1.18}
\journal{Journal of \LaTeX\ Templates}

\begin{document}

\begin{frontmatter}

\title{Dispersion-corrected compact finite difference discretizations for the elastic Helmholtz equation in two and three dimensions}

\author[mymainaddress]{Rachel Yovel\corref{mycorrespondingauthor}}\ead{yovelr@bgu.ac.il}
\cortext[mycorrespondingauthor]{Corresponding author}
\author[eliaddress]{Eli Turkel}\ead{turkel@post.tau.ac.il}
\author[mymainaddress]{Eran Treister}\ead{erant@cs.bgu.ac.il}

\address[mymainaddress]{Institute for Interdisciplinary Computational Science, the Stein Faculty of Computer and Information Science, Ben-Gurion University of the Negev, Beer-Sheva, Israel.}
\address[eliaddress]{School of Mathematical Sciences, Tel Aviv University, Tel Aviv, Israel}

\begin{abstract}

We develop a family of compact, dispersion-corrected marker-and-cell (MAC)-based finite-difference discretizations for the elastic Helmholtz equation. 
While dispersion correction is well established for scalar Helmholtz problems, fewer approaches exist for the elastic problems, and most use non-compact stencils.
To this end, we introduce a generalized MAC (GMAC) framework that enables compact discretization of the elastic Helmholtz equation. 
We prove that even though GMAC is non-symmetric, this discretization produces no numerical dissipation and no shear-mode separation. 
For certain weights, GMAC has lower polarization error than standard MAC.
We also introduce an elastic dispersion correction (EDC) process, extending real-shifted-wavenumber dispersion correction methods from acoustic to elastic Helmholtz problems. 
We prove and demonstrate numerically that GMAC is second-order accurate.
Furthermore, we show that EDC substantially reduces dispersion for standard MAC and can provide additional correction in some instances of GMAC.
Dispersion analysis shows that the resulting method outperforms existing compact finite-difference discretizations in dispersion accuracy and is competitive with some non-compact schemes. 
We give numerical evidence for improved phase alignment in mildly heterogeneous media.

\end{abstract}

\begin{keyword}

Acoustic Helmholtz equation, Elastic Helmholtz equation, wave propagation, finite difference discretization, numerical dispersion.

\footnote{\textbf{Funding:} This research was supported by The Israel Science Foundation (grant No. 656/23). 
RY is supported by the Ariane de Rothschild women doctoral program, and by the Kreitman Negev scholarship at BGU.
The authors also thank the Lynn and William Frankel Center at the Stein Faculty for Computer and Information Science at BGU.}

\end{keyword}

\end{frontmatter}

\section{Introduction}\label{sec:intro}

The Helmholtz equation is widely used for modeling frequency-domain wave propagation in acoustic or elastic media. Applications of frequency-domain wave modeling include acoustics \cite{wang2010acoustic, elec2001acoustic}, electromagnetic radiation \cite{zubeldia2012energy}, and seismic wave modeling \cite{wang20113d}. One application that has gained popularity is full-waveform inversion (FWI) \cite{pratt1999, virieux2009overview, JointEikFWI17}, which recovers the properties of the Earth's subsurface by solving an inverse problem to estimate the wave velocity distribution within a domain from boundary measurements.
In certain applications, such as brain tomography and void detection, optimal results require computing wave propagation using the elastic Helmholtz equation \cite{li2022target, sun2021deep, borisov20183d, kormann2017acceleration}.


Finite difference (FD) schemes are widely used to model seismic wave propagation in a finite domain.
The challenge of discretizing such systems increases with frequency: high-frequency waves typically require fine grids, which, in turn, lead to large linear systems. 
The \emph{pollution error} implies that the increase is faster than linear. 
For a second-order finite difference, $N\propto (kl)^\frac{3}{2}$ grid points are needed, where  $l$ is the size of the domain and $k$ is the wavenumber.
In practice, typically about 10--25 grid points per wavelength is a common choice, depending on the desired accuracy and the specifics of the scheme. 
Elastic waves, which comprise both shear and pressure waves, require even finer grids, since shear waves have a higher wavenumber \cite{martin2006marmousi2}.
Over the years, many methods have been developed for modeling waves in the time domain \cite{marfurt1984accuracy,carcione1993seismic, robertsson1994viscoelastic}, and in the frequency domain \cite{vstekl1998accurate, gosselin20143d}. 
Frequency-domain forward modeling in FWI offers advantages, as it avoids CFL constraints and provides a mathematically transparent framework for incorporating attenuation and analyzing its behavior.
However, to keep the resulting matrices small enough to solve, there is an ongoing effort to reduce the number of grid points per wavelength needed to achieve sufficient accuracy


The need to reduce numerical dispersion has motivated the development of many discretization schemes for the acoustic Helmholtz equation. 
Compact high-order FD discretizations for the Helmholtz equations \cite{singer1998high, turkel2013compact, sutmann2007compact} are known to yield less numerical dispersion compared to standard second-order stencils. 
In addition, compact schemes simplify the application of physical boundary conditions and can be applied efficiently on graphical processing units (GPUs). 
However, high-order schemes do not guarantee optimal numerical dispersion, especially at low numbers of grid points per wavelength.
To address this, \cite{operto20073d, aghamiry2022accurate} discretize the equation on combined rotated grids to reduce numerical dispersion, and \cite{stolk2016dispersion} use a ray method to minimize dispersion.
Some works suggest optimizing the weights of a compact stencil to reduce numerical dispersion; in \cite{babuvska1995generalized}, it is done in the context of a finite element method (FEM) on structured grids, and in \cite{jo1996optimal}, in the context of FD.
In \cite{cocquet2017finite}, a general compact stencil is equipped with an optimized real shifted wavenumber.
This method might be less practical in cases with a wide range of numbers of grid points per wavelength, 
despite its impressive phase alignment for as few as 2.5 grid points per wavelength, since all the weights need to be re-optimized for each case.
In a recent work, \cite{cocquet2024asymptotic}, an asymptotically optimal real shift is given, in closed form, to guarantee a minimal numerical dispersion for enough grid points per wavelength.
That method holds for 5 grid points per wavelength and more.
However, none of the above-mentioned methods is directly applicable to the elastic Helmholtz equation.



For elastic wave propagation, dispersion-reduction schemes based on different staggering strategies have been developed in the time domain or using stress-based formulations, including \cite{levander1988fourth, xie2024spatial, bahreman2026pure}.
The dispersion properties of elastic-wave discretizations have also been investigated using the boundary element method in \cite{seriani2008dispersion}.  
In the frequency domain, \cite{li2016Fourth} proposed a dispersion-reduced FD scheme for elastic wave propagation. Nevertheless, the resulting stencils are non-compact. 
In \cite{gu201321}, a suggested non-compact FD method for elastic wave propagation in the frequency domain is compared to a compact discretization, showing the need for $\times 3$ more grid points per wavelength in the compact case to achieve the same accuracy. 
Another compact finite-difference discretization for elastic wave modeling was proposed in \cite{gosselin20143d}.
Although these methods significantly improve phase alignment in elastic waves, they do not achieve a correction comparable to that of existing discretization methods for acoustic waves. 
The difficulty in reducing the numerical dispersion for elastic waves arises from the need to simultaneously align the shear and pressure waves, which have different wavenumbers. 
This difference becomes increasingly pronounced as the Poisson ratio increases. 
Motivated by these considerations, we aim to construct a compact discretization for the elastic Helmholtz equation whose dispersion properties are competitive with those of existing non-compact schemes.




In this work, we present two main contributions:
\begin{itemize}
\item a generalized MAC framework (GMAC) for the elastic Helmholtz equation that enables correcting the dispersion through the choice of weights, and
\item an elastic dispersion correction (EDC) process that further corrects the dispersion by simultaneously shifting the shear and pressure wavenumbers.
\end{itemize}
The GMAC framework provides additional degrees of freedom for controlling the discretization's dispersion properties, while the EDC extends wavenumber shifting to the elastic case by enabling simultaneous correction of shear and pressure waves. 
Together, these two components provide a flexible framework for reducing dispersion in elastic Helmholtz discretizations.


We analyze the dispersion of our method by sampling the Fourier symbols of the resulting saddle-point system. 
We compare our discretization to existing methods, and show that it achieves significantly better phase alignment compared to other compact methods.
Moreover, the dispersion error of our method is competitive with that of non-compact methods.
To show that our GMAC scheme is physically faithful despite being non-symmetric, we prove that it introduces no numerical dissipation.
Furthermore, in three dimensions, we prove that our discretization exhibits no shear-mode separation, and we show that its polarization error is lower than that of standard methods.
We prove analytically that the GMAC method is second-order accurate for any choice of weights, and demonstrate it numerically.
Moreover, the numerical results establish a further significant reduction of truncation error through our EDC process.
Although we developed this process for constant-coefficient problems, we also show numerically that it improves phase alignment in mildly heterogeneous media.


The paper is organized as follows: in Section~\ref{sec:acoustic}
we present the acoustic Helmholtz equation and existing discretization approaches.
In Section~\ref{sec:elastic} we present our first main contribution, and in Section \ref{sec:lame}, our second main contribution.
In Section~\ref{sec:analysis} we analyze numerical dispersion and develop general theory for the absence of numerical dissipation and shear-mode separation for this family of discretization methods.
In Section \ref{sec:numerical}, we demonstrate the numerical dispersion and polarization of the resulting discretizations, compare them to existing discretization methods, and test the efficiency of the dispersion correction in heterogeneous media.
In Section \ref{sec:OOA}, we analyze the order of accuracy of our method and numerically demonstrate the reduction of the truncation error for two-dimensional and three-dimensional media.
In Section \ref{sec:heterogeneous}, we give computational results for mildly heterogeneous media.
Finally, we discuss conclusions in Section~\ref{sec:conclusion}.

\section{Compact FD schemes for the acoustic Helmholtz equation}
\label{sec:acoustic}

In this section, we summarize the existing FD methods for the acoustic Helmholtz equation, which serves as a reference problem. This includes compact high-order schemes and dispersion-correction methods, which we use in Section \ref{sec:elastic} to design our methods for the elastic Helmholtz equation.

The acoustic Helmholtz equation, which models pressure waves in the frequency domain, is given by:
\begin{equation}\label{eq:acousitcHelm}
-\Delta p(\vec{x}) - k^2 p(\vec{x}) = q(\vec{x}), \quad \vec x \in \Omega,
\end{equation}
where $p = p(\vec{x})$ is the Fourier transform of the wave's pressure field, $q$ is the source of the waves, and $k = \kappa\omega$ is the wavenumber, with $\kappa$ the slowness of the medium (the inverse of the wave velocity) and $\omega=2\pi f$ the angular frequency.
When discretizing this equation, to mimic propagation in an infinite domain, it is commonly truncated to a finite domain equipped with absorbing boundary conditions (ABC) \cite{engquist1977absorbing} or perfectly matched layers (PML) \cite{berenger1994perfectly}.

\subsection{Compact schemes for acoustic Helmholtz in two dimensions}\label{subsec:acoustic2D}

A compact fourth-order FD scheme for the two-dimensional acoustic Helmholtz equation was developed in \cite{singer1998high}, based on a Pad{\'e} approximation.
In stencil notation, the discretization from \cite{singer1998high} reads

\begin{equation} \label{eq:helm4thCompact}
H_{4th} = -\frac{1}{h^2}
\begin{bmatrix}
\frac{1}{6} & \frac{2}{3}  & \frac{1}{6} \\[0.4em]
\frac{2}{3}  & -\frac{10}{3}  & \frac{2}{3}  \\[0.4em]
\frac{1}{6} & \frac{2}{3}  & \frac{1}{6}
\end{bmatrix}
- k^2
\begin{bmatrix}
 &  \frac{1}{12} & \\[0.4em]
 \frac{1}{12} & \frac{2}{3} &  \frac{1}{12} \\[0.4em]
 & \frac{1}{12} &
\end{bmatrix}.
\end{equation}
As a high-order stencil, \eqref{eq:helm4thCompact} demonstrates reduced numerical dispersion when the number of grid points per wavelength is moderate to high. 
This behavior is expected when considering the impact of dispersion error on the overall error in the acoustic Helmholtz equation. 
The truncation error consists of both dispersion and dissipation errors. In the time domain, dispersion error is associated with odd-order spatial derivatives, while dissipation error relates to even-order derivatives. 
In frequency-domain wave modeling, dissipation errors occur only when complex symbolic eigenvalues are present, which is not the case with the stencil \eqref{eq:helm4thCompact}. Without dissipation error, the truncation error equals the dispersion error, thereby minimizing dispersion in high-order methods. However, this holds as $h\to0$ for high-order stencils but does not guarantee optimal dispersion with a low number of grid points per wavelength.

To address dispersion errors with fewer grid points per wavelength, researchers have developed specialized dispersion-correction methods. 
We present here some of these methods, focusing on compact stencils.
We discuss the method from \cite{jo1996optimal}, which proposes a global choice of weights for a compact stencil, and \cite{cocquet2017finite}, which improves accuracy by optimizing stencil weights for each number of grid points per wavelength and incorporating a real shift as a dispersion corrector.
In \cite{cocquet2024asymptotic}, a fourth-order method is equipped with an asymptotically calculated modified wavenumber, given by a closed-form expression, which achieves excellent dispersion correction for as few as 5 grid points per wavelength, while being less efficient for extremely low numbers of grid points per wavelength.

In more detail, in \cite{jo1996optimal},
a general symmetric 9-point stencil discretization \eqref{eq:acousitcHelm}
is described in terms of three variables $a,b$ and $c$:
\begin{equation}\label{eq:JSS_sten}
H_{9pt} = -\frac{1}{h^2}\left(a \begin{bmatrix}
 & 1 & \\
1 & -4 & 1 \\
 & 1 &
\end{bmatrix}
+ \frac{1-a}{2}
\begin{bmatrix}
1 &  & 1\\
 & -4 &  \\
1 &  & 1
\end{bmatrix}\right)
- k^2 \begin{bmatrix}
\frac{1-b-c}{4} & \frac{c}{4} & \frac{1-b-c}{4} \\[0.5em]
\frac{c}{4} & b & \frac{c}{4} \\[0.5em]
\frac{1-b-c}{4} & \frac{c}{4} & \frac{1-b-c}{4}
\end{bmatrix}.
\end{equation}
Generally, the stencil \eqref{eq:JSS_sten} is second-order accurate.
However, some weights can achieve fourth- or sixth-order accuracy; e.g., $[a,b,c] = [2/3,2/3,1/3]$ yields the fourth-order stencil \eqref{eq:helm4thCompact}.
Subsequently, in \cite{jo1996optimal}, $a,b$ and $c$ are chosen to minimize the numerical dispersion in the following sense: 
the relative dispersion error is integrated over the propagation directions and over different numbers of grid points per wavelength $G$, and minimized in the $L^2$ norm. For $G\geq 4$, the authors reported the following values $[a,b,c] = [0.5461,0.6248,0.37524]$ as optimal weights in terms of numerical dispersion.

In \cite{cocquet2017finite},  a further dispersion correction is suggested, based on a stencil of the same form \eqref{eq:JSS_sten}, equipped with a real-shifted wavenumber to align the phases.
The authors present an optimization process that chooses $[a,b,c]$ and a modified wavenumber $k^*$ for each different value of $G$. 
For $G=4$, for example, the values $[a,b,c] = [0.6054,1.0532,0.0002]$ and $k^* = 0.87725k$ are reported as optimal, and for $G=2.5$, the optimal values are $[a,b,c] = [0.5324,1.0553,0.0003]$ and $k^* = 0.72186k$. 

In summary, the presented dispersion-correction compact discretizations in two dimensions offer distinct advantages: 
Some achieve higher accuracy and require no re-optimization of weights, while others achieve excellent alignment with very few grid points per wavelength.
The mechanism we develop later enables us to apply the same ideas to achieve a similar dispersion correction for the elastic Helmholtz equation in the mixed formulation using only compact stencils.

\subsection{Compact schemes for acoustic Helmholtz in three dimensions}\label{subsec:acoustic3D}

We derive a three-dimensional generalization of the acoustic Helmholtz discretization \eqref{eq:helm4thCompact} using Pad{\'e} approximations.
This is based on previous work \cite{harari1995accurate, turkel2013compact}, in which the resulting discretization was not explicitly given in a stencil notation.
For completeness, we present the full derivation in the following.

For ease of notation, we use $\delta_x,\delta_y$ and $\delta_z$ to denote discrete derivatives in the $x_1,x_2$ and $x_3$ directions, respectively.
Define $\delta_{xx}$ as the second-order divided difference in the $x$ direction with a grid size $h$ and with
a similar notation in the $y$ and $z$ directions.
Applying a Pad{\'e} fourth-order approximation to the acoustic Helmholtz equation \eqref{eq:acousitcHelm} in three dimensions, gives
\begin{equation}\label{eq:pade}
  \frac{\delta_{xx}}{1+\frac{h^2}{12} \delta_{xx}} u
+ \frac{\delta_{yy}}{1+\frac{h^2}{12} \delta_{yy}} u
+ \frac{\delta_{zz}}{1+\frac{h^2}{12} \delta_{zz}} u
+ k^2 u = f.
\end{equation}
Multiplying by the product of the denominators and simplifying, we obtain
\begin{align}\label{eq:harari}
&\left[ \left(\delta_{xx} + \delta_{yy} + \delta_{zz}\right)
+ \frac{h^2}{6} \left(\delta_{xx}\delta_{yy} + \delta_{xx}\delta_{zz} + \delta_{yy}\delta_{zz}\right) + \frac{3 h^4}{144} \delta_{xx}\delta_{yy}\delta_{zz} \right] u \\
&+ \left[ 1 + \frac{h^2}{12} \left(\delta_{xx} + \delta_{yy} + \delta_{zz}\right)
+ \frac{h^4}{144} \left(\delta_{xx}\delta_{yy} + \delta_{xx}\delta_{zz} + \delta_{yy}\delta_{zz}\right)
+ \frac{h^6}{1738} \delta_{xx} \delta_{yy} \delta_{zz} \right] k^2 u \nonumber \\
&= \left[ 1 + \frac{h^2}{12} \left(\delta_{xx} + \delta_{yy} + \delta_{zz}\right)
+ \frac{h^4}{144} \left(\delta_{xx}\delta_{yy} + \delta_{xx}\delta_{zz} + \delta_{yy}\delta_{zz}\right)
+ \frac{h^6}{1738} \delta_{xx} \delta_{yy} \delta_{zz} \right] f. \nonumber
\end{align}
By \eqref{eq:harari}, the fourth-order accurate Helmholtz operator is composed of the Laplacian
\begin{equation}
\Delta_{4th} = \left(\delta_{xx} + \delta_{yy} + \delta_{zz}\right) +\frac{h^2}{6} \left(\delta_{xx}\delta_{yy} + \delta_{xx}\delta_{zz} + \delta_{yy}\delta_{zz}\right) +O(h^4),
\end{equation}
and the mass term
\begin{equation}
M_{4th} = \left[ 1+ \frac{h^2}{12}\left(\delta_{xx}+\delta_{yy}+\delta_{zz}\right) \right] k^2 + O(h^4),
\end{equation}
where, to keep the fourth order of accuracy, the forcing term $f$ on the right-hand side must be discretized similarly to the mass.

Next, we translate the resulting Laplacian and mass expressions to stencil notation.
By definition,
\begin{equation} \label{eq:stencilLapStandard}
\delta_{xx} + \delta_{yy} + \delta_{zz} = \Delta_{standard} =
\frac{1}{h^2}\left[
\begin{bmatrix}
 &  & \\
 &1& \\
 &  &
\end{bmatrix}
\begin{bmatrix}
 &1& \\
1&-6&1\\
 &1&
\end{bmatrix}
\begin{bmatrix}
 &  & \\
 &1& \\
 &  &
\end{bmatrix}
\right].
\end{equation}
To calculate fourth-order derivatives, we apply
\begin{equation}
\delta_{xx}\delta_{yy} \phi_{i,j} = \delta_{xx} \left(\frac{\phi_{i,j-1} -2\phi_{i,j}+\phi_{i,j+1}}{h^2} \right) = \frac{1}{h^4}
\begin{bmatrix}
1 & -2 & 1 \\
-2 & 4 & -2 \\
1 & -2 & 1
\end{bmatrix},
\end{equation}
and hence, from symmetry reasons, we obtain
\begin{align}
 \frac{h^2}{6} \left(\delta_{xx}\delta_{yy} + \delta_{xx}\delta_{zz} + \delta_{yy}\delta_{zz}\right)& = 
 \frac{1}{6h^2}\left[
\begin{bmatrix}
 &  & \\
 & & \\
 &  &
\end{bmatrix}
\begin{bmatrix}
1 & -2 & 1 \\
-2 & 4 & -2 \\
1 & -2 & 1
\end{bmatrix}
\begin{bmatrix}
 &  & \\
 & & \\
 &  &
\end{bmatrix}
\right] + \\ \nonumber
& + \frac{1}{6h^2}\left[
\begin{bmatrix}
 &  & \\
1 & -2& 1\\
 &  &
\end{bmatrix}
\begin{bmatrix}
 &  &  \\
-2 & 4 & -2 \\
 &  &
\end{bmatrix}
\begin{bmatrix}
 &  & \\
1 & -2& 1\\
 &  &
\end{bmatrix}
\right] + \\ \nonumber
& + \frac{1}{6h^2}\left[
\begin{bmatrix}
 & 1 & \\
 & -2& \\
 &  1&
\end{bmatrix}
\begin{bmatrix}
 & -2 &  \\
 & 4 &  \\
 & -2 &
\end{bmatrix}
\begin{bmatrix}
 & 1 & \\
 & -2& \\
 &  1&
\end{bmatrix}
\right] = \\ \nonumber
& = \frac{1}{6h^2}\left[
\begin{bmatrix}
 & 1 & \\
 1& -4&1 \\
 &  1&
\end{bmatrix}
\begin{bmatrix}
 1& -4 &1  \\
-4 & 12 & -4  \\
1 & -4 & 1
\end{bmatrix}
\begin{bmatrix}
 & 1 & \\
1 & -4& 1\\
 &  1&
\end{bmatrix}
\right].
\end{align}
Summing the above calculated terms, we get the 19-point Laplacian
\begin{equation} \label{eq:stencilLaplacian3D}
\Delta_{4th} = \frac{1}{6h^2}
\left[
\begin{bmatrix}
 & 1 & \\
 1 & 2 & 1 \\
 &  1 &
\end{bmatrix}
\begin{bmatrix}
 1 & 2 & 1  \\
2 & -24 & 2  \\
1 & 2 & 1
\end{bmatrix}
\begin{bmatrix}
 & 1 & \\
 1 & 2 & 1 \\
 &  1 &
\end{bmatrix}
\right]
\end{equation}
and the 7-point mass
\begin{equation} \label{eq:stencilMass3D}
M_{4th} = \frac{1}{12}
\left[
\begin{bmatrix}
 &  & \\
 & 1 &  \\
 &   &
\end{bmatrix}
\begin{bmatrix}
 & 1 &   \\
1 & 6 & 1  \\
  & 1 &
\end{bmatrix}
\begin{bmatrix}
 &  & \\
 & 1 &  \\
 &   &
\end{bmatrix}
\right] k^2,
\end{equation}
which forms the fourth-order accurate discretization of the acoustic Helmholtz operator in three dimensions.
Similarly to the two-dimensional case, a sixth-order version can be constructed using the same Laplacian stencil with a 19-point mass stencil.
However, our experience shows that the sixth-order scheme does not significantly reduce numerical dispersion, so we focus on generalizing the fourth-order scheme, equipped with dispersion correction techniques, to the elastic Helmholtz equation.
As we show below, the elastic generalization is not fourth-order accurate; nevertheless, this scheme can significantly reduce the dispersion error.

\section{A generalized MAC framework for the elastic Helmholtz equation (GMAC)}\label{sec:elastic}

In this section we present one of our main contributions:
a generalized MAC framework for the discretization of the elastic Helmholtz equation.
This generalization enables compact stencils with various weights, extending many dispersion-correction discretization schemes from the acoustic to the elastic Helmholtz equation

The elastic version of the Helmholtz equation in an isotropic medium with no attenuation is given by
\begin{equation}\label{eq:elasticHelm}
\grad\lambda\div\vec u + \vec\div\mu\left(\vec\grad\vec u+\vec\grad\vec u^T\right) +\rho \omega^2  \vec u = \vec q_{s}(\vec{x})
\end{equation}
or by the following formulation, which is equivalent in the case of constant coefficients,
\begin{equation}\label{eq:elasticHelmEquivalent}
\grad(\lambda + \mu)\div\vec u + \vec\div\mu\vec\grad\vec u +\rho \omega^2 \vec u = \vec q_{s}(\vec{x}),
\end{equation}
where $\vec u(\vec x)$ is the vector of displacements, $\lambda$ and $\mu$ are the Lam{\'e} coefficients determining the stress-strain relationship in the medium, $\rho$ is the density, and $\vec{q}_s(\vec x)$ is the source.
Similar to the acoustic case, $\omega$ is the angular frequency.
The pressure and shear wave velocities and their corresponding wavenumbers are given by
\begin{equation}\label{eq:ContVelocities}
V_p = \sqrt{\frac{\lambda+2\mu}{\rho}} \quad \text{and} \quad V_s = \sqrt{\frac{\mu}{\rho}}, \qquad k_p = \frac{\omega}{V_p} \quad \text{and} \quad k_s = \frac{\omega}{V_s}.
\end{equation}
respectively (see, e.g., \cite{li2016Fourth}).
Pressure waves are always faster (have a lower wavenumber or slowness), and in nearly incompressible media $ V_p$ is significantly greater than $ V_s$.
The Poisson ratio is defined by
$\sigma = \frac{\lambda}{2(\lambda+\mu)},$
and the nearly incompressible case is characterized by $\sigma\to 0.5$. 

The mixed formulation for the elastic Helmholtz equation was suggested in \cite{treister2024hybrid}, following the mixed formulation for linear elasticity introduced in \cite{gaspar2008distributive}. 
It gives the elastic Helmholtz equation a saddle-point form that couples displacement and pressure.
To obtain the mixed formulation, a new pressure\footnote{The variable $p$ has units of pressure, similarly to $p$ from the acoustic Helmholtz equation \eqref{eq:acousitcHelm}. 
In fact, assuming constant coefficients, one can derive the acoustic equation for $p$ from the elastic \eqref{eq:elasticMixed} by taking the divergence of the first equation.} variable is introduced $p = -(\lambda+\mu)\div\vec{u},$   yields
\begin{equation} \label{eq:elasticMixed}
\begin{pmatrix}
\vec\div\mu\vec\grad + \rho \omega^2 & -\grad \\
\div & \frac{1}{\lambda + \mu} I
\end{pmatrix}
\begin{pmatrix}
 \vec u \\
 p
\end{pmatrix}
=
\begin{pmatrix}
\vec q_{s} \\
0
\end{pmatrix},
\end{equation}
where $I$ is the identity operator.
Since \eqref{eq:elasticHelmEquivalent} is simply the primal Schur complement of \eqref{eq:elasticMixed}, any discretization scheme for \eqref{eq:elasticMixed} can be used to discretize \eqref{eq:elasticHelmEquivalent}.

\begin{figure}
\begin{center}	\newcommand{\image}[1]{\includegraphics[width=0.25\linewidth]{./#1}}
    \subfigure[\footnotesize MAC locations in 3D]{\image{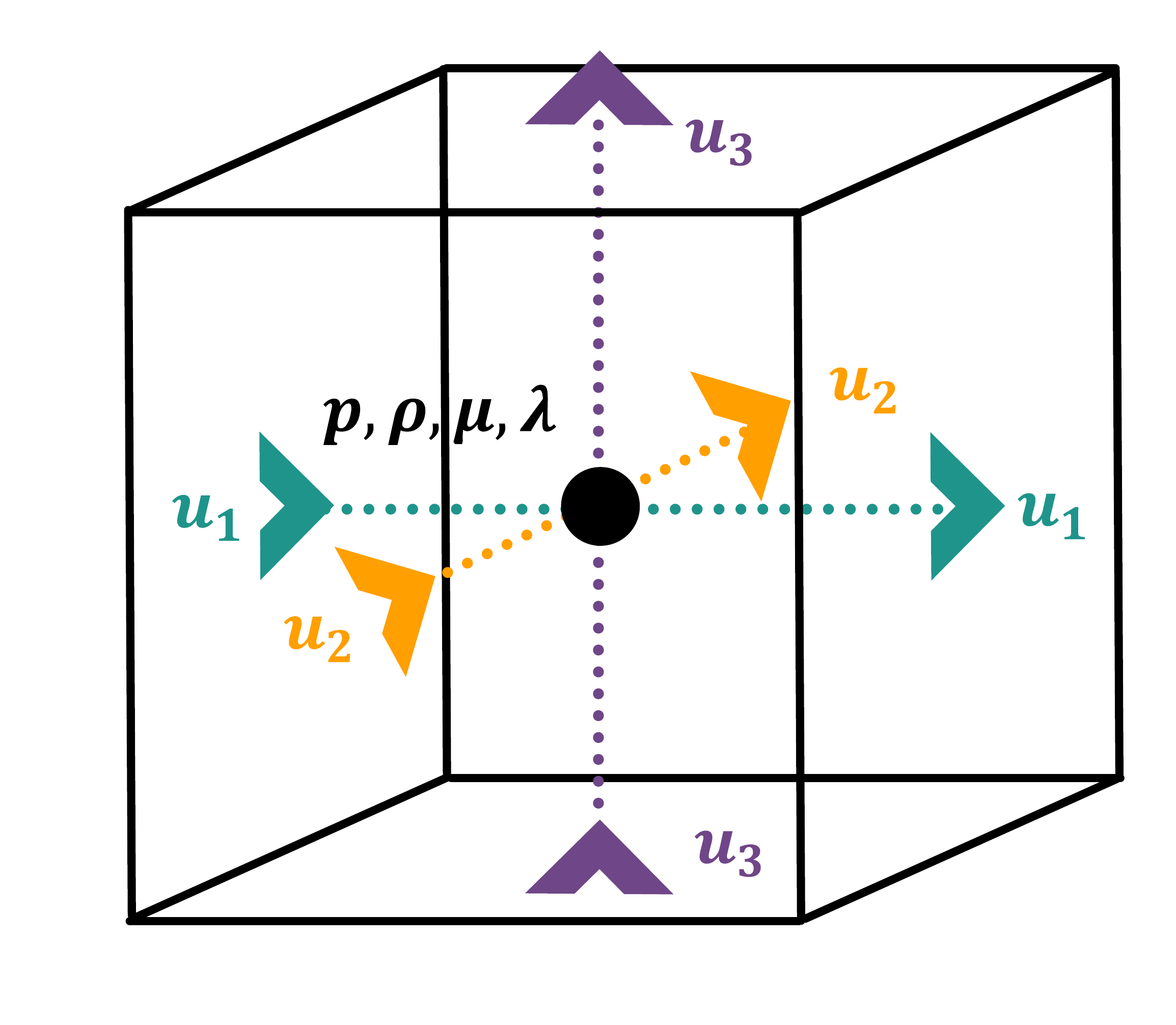}\label{fig:cell}}
    \hspace{30pt}
    \subfigure[\footnotesize $\grad$ and $\div$ stencils in 2D ]{\image{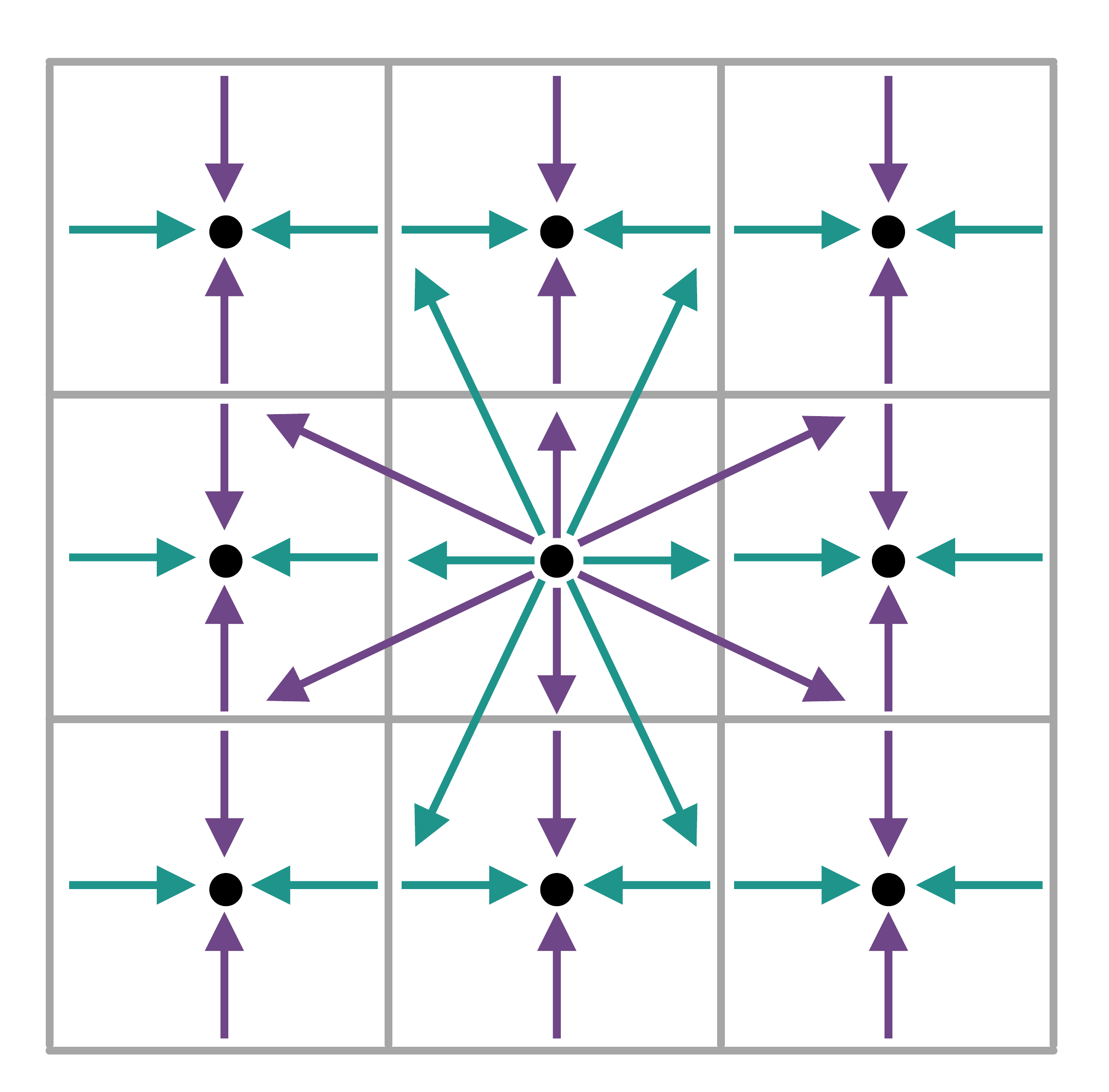}\label{fig:dxdystencils}}\\
\end{center}
\caption{On the left: MAC scheme locations in three dimensions; displacements are located on the faces, and pressure and physical variables are in the cell centers. On the right: 6-point gradient and standard 2-point divergence in two dimensions (in turquoise for $\partial_{x_1}$ and in purple for $\partial_{x_2}$) that result in a 9-point Laplacian.
}\label{fig:discretization}
\end{figure}

The saddle-point structure of \eqref{eq:elasticMixed} motivates discretizing the system by MAC-based \cite{harlow1965numerical} staggered grid schemes:
the displacement components $\bfu_1,\bfu_2$ and $\bfu_3$ are located on the faces, whereas the pressure $\bfp$ is located at the cell centers, see Figure \ref{fig:cell}.
While a standard MAC discretization uses standard second-order central-difference operators at these grid locations, we generalize it by keeping the same locations and modifying the weights.
In the case of non-constant parameters (which can come from heterogeneous media or absorbing boundary conditions), we locate the physical variables
$\bfgamma,\bfrho,\bflambda,\bfmu$ in the cell centers.

In the next subsections, we develop a mechanism that allows a general choice of stencils for the first derivatives that keeps the resulting Laplacian compact, and a general compact choice of the mass term. 
In two dimensions, the discretization suggested in \cite{yovel2024lfa} is a special case of GMAC, while the additional degrees of freedom in the mass term that we present below enable a further dispersion reduction. In three dimensions, our framework yields a new compact discretization.

\subsection{Construction of elastic generalized MAC schemes (GMAC) in two dimensions}\label{subsec:schemes2D}

In this subsection, we present GMAC, a family of generalized MAC methods for discretizing the elastic Helmholtz equation in mixed formulation \eqref{eq:elasticMixed}.
It generalizes standard MAC by enabling any compact stencil for the Laplacian and mass, with compatible stencils for the divergence and the gradient.
Numerical dispersion can be reduced significantly by enabling a general choice of weights.

More explicitly, the Laplacian stencil from \eqref{eq:helm4thCompact} is the convolution of the following gradient and divergence. 
Let the first derivative weighted by $a$ be
\begin{equation} \label{eq:dxSpread}
(\partial_{x_1})^a_{h/2} =          a \left(\partial_{x_1}\right)^{standard} + 
         (1-a) \left(\partial_{x_1}\right)^{spread}
\end{equation}
 where
 \begin{equation}
     \left(\partial_{x_1}\right)^{standard} = 
       \frac{1}{h}  \begin{bmatrix}
        -1 & * & 1 \\
    \end{bmatrix}
    \quad \text{and} \quad
    \left(\partial_{x_1}\right)^{spread} =
    \frac{1}{4h}\begin{bmatrix}
        -1 & & 1\\
        -2 & * & 2 \\
        -1 & & 1
    \end{bmatrix}.
\end{equation}
Then the gradient and the divergence are
\begin{equation} \label{eq:GradAndDiv}
\nabla_h^{a} =
\begin{pmatrix}	(\partial_{x_1})^a_{h/2} \\[0.5em]
(\partial_{x_2})^a_{h/2}
\end{pmatrix}
\quad \text{ and } \quad
(\nabla\cdot)_h =
\begin{pmatrix}
(\partial_{x_1})_{h/2} & (\partial_{x_2})_{h/2}
\end{pmatrix}.
\end{equation}
Fig. \ref{fig:dxdystencils} demonstrates the compactness of the resulting stencil using a non-symmetric choice: 
If we take a 2-point divergence and 6-point gradient, we obtain a 9-point Laplacian in the locations of their composition.

Let $M^{b,c}$ be a mass term similar to that of \eqref{eq:JSS_sten}, only with the (acoustic) wavenumber $k^2$ replaced by $\rho\omega^2$.
As a result, we present the following \textbf{GMAC (generalized marker and cell) discretization} of \eqref{eq:elasticMixed}:
\begin{equation} \label{eq:MAC-dispersion}
\mathcal{H}^{[a,b,c]}
\begin{pmatrix}
 \vec \bfu \\
 \bfp
\end{pmatrix}
=
\begin{pmatrix}
A & B_1^T \\
B_2 & -C
\end{pmatrix}
\begin{pmatrix}
 \vec \bfu \\
 \bfp
\end{pmatrix}
=
\begin{pmatrix}
-\vec\grad^T_h A_e(\bfmu) \vec\grad^a_h +  M^{b,c} A_f(\bfrho) & (\grad\cdot)_h^T \\
(\grad\cdot)_h^a & \text{diag}\left(\frac{1}{\bflambda+\bfmu}\right)
\end{pmatrix}
\begin{pmatrix}
 \vec \bfu \\
 \bfp
\end{pmatrix}
=
\begin{pmatrix}
\vec \bfq_{s} \\
\mathbf{0}
\end{pmatrix},
\end{equation}
where $A_e(\bfmu), A_f(\bfrho)$ are diagonal matrices with the edge-averaged $\bfmu$ values and face-averaged $\bfrho$ values on their diagonals.
We note that the two-dimensional discretization in \cite{yovel2024lfa} is recovered as a special case of the GMAC discretization in \eqref{eq:MAC-dispersion}, with $a=b=2/3$ and $c=1-b=1/3$.
Although this choice yields less numerical dispersion than the standard MAC, the more general framework suggested here allows for further dispersion reduction and adapts acoustic dispersion-correction methods (such as \cite{jo1996optimal}) to the elastic Helmholtz equation.

\subsection{Elastic generalized MAC schemes (GMAC) in three dimensions}\label{subsec:schemes3D}

We next generalize the discussion to three dimensions in three steps:
First, we present the Laplacian \eqref{eq:stencilLaplacian3D}
as a convex combination of a standard three-dimensional Laplacian stencil and a ``skew'' Laplacian stencil.
Second, we develop first-derivative stencils whose Kronecker product yields these Laplacians.
Lastly, we present a general symmetric 19-point mass stencil analogous to the mass in \eqref{eq:JSS_sten}, from which the 7-point mass stencil \eqref{eq:stencilMass3D} is accepted as a special case.

For the first step, we note that the fourth-order Laplacian \eqref{eq:stencilLaplacian3D} is a convex combination of the standard second-order stencil \eqref{eq:stencilLapStandard} and the following ``skew'' Laplacian stencil 
\begin{equation} \label{eq:stectilLapSkew}
\Delta_{skew}
=
\frac{1}{4h^2}
\left[
\begin{bmatrix}
 & -1 & \\
-1 &  & -1 \\
 &  -1 &
\end{bmatrix}
\begin{bmatrix}
-1 &  & -1  \\
 & 12 &   \\
-1  &  & -1
\end{bmatrix}
\begin{bmatrix}
 & -1 & \\
-1 &  & -1 \\
 &  -1 &
\end{bmatrix}
\right].
\end{equation}
A general symmetric 19-point Laplacian stencil is hence of the form
\begin{equation}\label{eq:Lap-19-pt-sten}
\Delta_{19pt}=
\frac{1}{h^2}
\left[
\begin{bmatrix}
 & \frac{1-a}{12} & \\[0.5em]
\frac{1-a}{12} & \frac{a}{6} & \frac{1-a}{12} \\[0.5em]
 &  \frac{1-a}{12} &
\end{bmatrix}
\begin{bmatrix}
\frac{1-a}{12} & \frac{a}{6} & \frac{1-a}{12}  \\[0.5em]
\frac{a}{6} & -4 &  \frac{a}{6} \\[0.5em]
\frac{1-a}{12} & \frac{a}{6} & \frac{1-a}{12}
\end{bmatrix}
\begin{bmatrix}
 & \frac{1-a}{12} & \\[0.5em]
\frac{1-a}{12} & \frac{a}{6} & \frac{1-a}{12} \\[0.5em]
 &  \frac{1-a}{12} &
\end{bmatrix}
\right],
\end{equation}
where the special case $a=1/3$ yields the Laplacian stencil \eqref{eq:stencilLaplacian3D}.

For the second step, we suggest first derivatives whose combination leads to the Laplacian \eqref{eq:Lap-19-pt-sten}.
The standard first derivative $(\partial_{x_1})_{h/2} = \frac{1}{h} \begin{bmatrix} -1 & * & 1\end{bmatrix}$, together with the similarly defined standard first derivatives in the other directions for $(\partial_{x_2})_{h/2}$ 
and $(\partial_{x_3})_{h/2}$, lead to the standard gradient and divergence, and their product is the standard Laplacian \eqref{eq:stencilLapStandard}.
Using the standard first derivatives for the divergence and the following spread first derivatives for the gradient,
\begin{equation}\label{eq:del_spread}
(\partial_{x_1})_{h/2}^{spread} =
\frac{1}{8h}
\left[
\begin{bmatrix}
 &  & \\
-1 &  & 1 \\
 &   &
\end{bmatrix}
\begin{bmatrix}
-1 &  & 1  \\
-4 & * & 4  \\
-1  &  & 1
\end{bmatrix}
\begin{bmatrix}
 &  & \\
-1 &  & 1 \\
 &   &
\end{bmatrix}
\right]
\end{equation}
(and similarly for the $x_2$ and $x_3$ directions), we get the gradient and divergence
\begin{equation}\label{eq:GradAndDiv3D}
\grad_h^a = \begin{pmatrix}
(\partial_{x_1})_{h/2}^a \\[0.5em]
(\partial_{x_2})_{h/2}^a \\[0.5em]
(\partial_{x_3})_{h/2}^a
\end{pmatrix}
\qquad \text{and} \qquad
(\div)_h = \begin{pmatrix}
(\partial_{x_1})_{h/2} & (\partial_{x_2})_{h/2} & (\partial_{x_3})_{h/2}
\end{pmatrix},
\end{equation}
where $(\partial_{x_1})_{h/2}^a = a(\partial_{x_1})_{h/2}^{standard}+(1-a)(\partial_{x_1})_{h/2}^{spread}$.
Indeed, $(\div)^a_h \grad_h = \Delta_{19pt}$, which is what we would get in the leading block in \eqref{eq:MAC-dispersion} for $\mu=1$.

For the third and last step, we suggest a three-dimensional generalization of the mass stencil in \eqref{eq:JSS_sten}.
The generalized stencil we suggest is a 19-point stencil (and not a 27-point stencil), to keep the analogy to the two-dimensional case accurate.
In two dimensions, the 9-point mass stencil in \eqref{eq:JSS_sten} is comprised of the center of the stencil with the weight $b$, locations 
with one coordinate offset from the center, dividing the weight $c$ 
between them, and locations with two coordinates offset from the center of the stencil, dividing the remaining $1-b-c$ between them.
Based on this understanding, we define a general 19-point stencil for the mass as
\begin{equation}\label{eq:Mass-19-pt-sten}
M^{b,c} =
\left[
\begin{bmatrix}
 & \frac{1-b-c}{12} & \\[0.5em]
\frac{1-b-c}{12} & \frac{c}{6} & \frac{1-b-c}{12} \\[0.5em]
 &  \frac{1-b-c}{12} &
\end{bmatrix}
\begin{bmatrix}
\frac{1-b-c}{12} & \frac{c}{6} & \frac{1-b-c}{12}  \\[0.5em]
\frac{c}{6} & b & \frac{c}{6} \\[0.5em]
\frac{1-b-c}{12}  & \frac{c}{6} & \frac{1-b-c}{12}
\end{bmatrix}
\begin{bmatrix}
 & \frac{1-b-c}{12} & \\[0.5em]
\frac{1-b-c}{12} & \frac{c}{6} & \frac{1-b-c}{12} \\[0.5em]
 &  \frac{1-b-c}{12} &
\end{bmatrix}
\right].
\end{equation}
We observe that the weights $[b,c]=[1/2,1/2]$ lead to the mass stencil \eqref{eq:stencilMass3D}.
Finally, a GMAC discretization $\mathcal{H}^{[a,b,c]}$ to the saddle-point system \eqref{eq:elasticMixed} is defined according to \eqref{eq:MAC-dispersion}.

\section{Elastic dispersion correction (EDC): simultaneous correction for shear and pressure waves}\label{sec:lame}

In this section we present another main contribution:
a general framework for elastic dispersion correction (EDC).
The process we describe below enables us to correct shear and pressure wavenumbers simultaneously by shifting $\omega$ and $\lambda$.
Together with the GMAC scheme presented in Section \ref{sec:elastic}, this shifting method enables a further dispersion correction and completes the extension of any compact dispersion correction method from the acoustic to the elastic Helmholtz equation.

Some acoustic dispersion-correction methods are based on the idea of
using $k^*\neq k$ to discretize the acoustic Helmholtz equation \eqref{eq:acousitcHelm}.
Applying this idea in a naive way to the elastic Helmholtz equation in mixed formulation \eqref{eq:elasticMixed} reads:
replace the angular frequency by $\omega^*\neq \omega$.
Nevertheless, the number of grid points per \emph{shear} wavelength $G_s$ is not equal to the number of grid points per \emph{pressure} wavelength, $G_p$.
In the context of acoustic Helmholtz, different numbers of grid points per wavelength require different adapted wavenumbers $k^*$, even when using the same stencil weights (see, e.g., \cite{cocquet2024asymptotic}).
Hence, in the context of elastic Helmholtz, naive application of $\omega^*\neq\omega$ 
may correct the shear waves while creating misalignment in the pressure waves.

We suggest a mechanism to impose these different adapted frequencies to the shear and pressure waves.
Let $t_p$ be the constant corresponding to the correction of pressure waves, and $t_s$ the constant corresponding to the correction of shear waves.
That is, for the shear and pressure wavenumbers $k_p$ and $k_s$ from \eqref{eq:ContVelocities}, we define $t_p$ and $t_s$ to be the optimal corrections for the corresponding acoustic Helmholtz equations \eqref{eq:acousitcHelm}.
Namely, $t_s$ is chosen such that $k_s^* = t_s k$ is an optimal discrete wavenumber to approximate the acoustic Helmholtz equation with a continuous wavenumber $k_s$, and similarly for $t_p$.

The first step of correction is setting $\omega^* = t_s \omega$.
This step sets $k_s^* = \frac{t_s\omega}{V_s} = t_sk_s$, inevitably setting
$k_p^* = t_sk_p,$ which is not the desired correction.
Our aim in the second step is to replace $t_s$ by $t_p$ in the correction for the pressure waves, so that setting
$k_p^{**} = \frac{t_p}{t_s}k_p^*$
fixes the wrong correction.
To achieve this, we modify the Lam{\'e} coefficient $\lambda,$ which determines the pressure velocity only, by another constant.
We look for $s$ that will satisfy
\begin{equation}\label{eq:s_eq}
k_p^{**} = \frac{\omega^*}{V_p^*}=
\frac{\omega^*\sqrt{\rho}}{\sqrt{\lambda s + 2\mu}} = \frac{t_p}{t_s}\frac{\omega^*\sqrt{\rho}}{\sqrt{\lambda+2\mu}}
= \frac{t_p}{t_s}\frac{\omega^*}{V_p} =  \frac{t_p}{t_s}k_p^*.
\end{equation}
The solution to the third equality in \eqref{eq:s_eq} is 
\begin{equation}\label{eq:s_def}
s = \frac{t_s^2}{t_p^2} + \left(\frac{t_s^2}{t_p^2}-1\right)\frac{2\mu}{\lambda}.
\end{equation}
and placing $\lambda^* = s\lambda$ completes the correction process.

We summarize \textbf{our EDC process} in the following steps:
\begin{enumerate}
\item Replace $\omega$ by $\omega^* = t_s\omega$ where $t_s$ is shift for optimal dispersion correction for $G_s$.\;
\item Calculate $s$ by \eqref{eq:s_def} using $t_p,$ the optimal shift for  dispersion correction for $G_p$.
\item Replace $\lambda^*=s\lambda$.
\end{enumerate}
Note that the use of this algorithm requires knowledge of the appropriate shifting constants for the acoustic Helmholtz equation with shear and with pressure wave velocity.
These shifting constants can be either optimized by sampling symbols in the Fourier space or taken based on known values for acoustic dispersion correction methods.
For example, one can use specific shifts from \cite{cocquet2017finite} or an asymptotic closed-form formula for the shift, proved in \cite{cocquet2024asymptotic}, that depends on the number of grid points per wavelength.
The process presented here can be seen as a reduction from existing dispersion-correction methods for the acoustic Helmholtz equation to the elastic one.
When the Poisson ratio and the corresponding $G_p$ are sufficiently large, a correction is not required to the pressure waves.
In such a case, setting $t_p=1$ in \eqref{eq:s_def} suffices.
Yet, even in this case, one has to reverse the misalignment caused by the shear wave correction by setting $\lambda^*=s\lambda$.


\section{Symbol analysis: dispersion, dissipation and polarization}\label{sec:analysis}


In this section, we first recall the definition of numerical dispersion in the scalar case and generalize it to a dispersion analysis framework for shear and pressure waves, based on a discretization of the saddle-point formulation \eqref{eq:elasticMixed}.
Second, we recall the definition of numerical dissipation and shear mode separation and prove that our GMAC discretization \eqref{eq:MAC-dispersion} does not produce numerical dissipation (despite being non-symmetric), and that the two shear modes in three dimensions always coincide.
Finally, we discuss polarization error, another source of error relevant only to vector equations, which we need to consider here.
We recall the definition of pressure polarization error and suggest an analogous definition for shear polarization error, which we later use in the numerical examples in Section \ref{sec:numerical}.

\subsection{Numerical dispersion for scalar equations}\label{subsec:dispersion_scalar}

Roughly speaking, \emph{numerical dispersion} describes the difference between the continuous and discrete wave velocity.
First, we define the set of values in different propagation directions.

\begin{definition}[dispersion relation, scalar equations]\label{def:disp_rel}
Let $A$ be an operator (continuous or discrete), over a space of dimension $d$.
The \emph{dispersion relation} of $A$ is the kernel of the symbol:
\begin{equation}\label{eq:disp_rel}
\left\{\xi\in\mathbb{R}^d \mid \widetilde{A}(\xi) = 0\right\}
\end{equation}
where $\widetilde{A}$ is the Fourier symbol of $A$.
\end{definition}

For the (scalar) continuous Helmholtz equation \eqref{eq:acousitcHelm}, the symbol is calculated by applying the continuous operator to a plane wave $e^{i\xi\cdot x}$,
 yielding the dispersion relation $\{\xi\in\mathbb{R}^d \mid \|\xi\|^2=k^2\}$,
that is, a circle of radius $k$.
For a discrete operator defined as a convolution against a stencil, representing a scalar equation, the symbol is calculated using the stencil weights, e.g., in two dimensions,
\begin{equation}\label{eq:general_stencil}
S = \begin{bmatrix}
s_{-1,1} & s_{0,1} & s_{1,1} \\
s_{-1,0} & s_{0,0} & s_{1,0} \\
s_{-1,-1} & s_{0,-1} & s_{1,-1}
\end{bmatrix}
\quad \Rightarrow \quad
\widetilde{A}_h(\theta) = \sum s_{l,m} e^{i (l\theta_1+m\theta_2)},
\end{equation}
where $\theta$ is the dimensionless Fourier variable $\theta = \xi h$,
see \cite{cocquet2017finite}.
When we compare the discrete and continuous phase velocities via the dispersion relation, we translate the continuous case to the dimensionless variable $\theta$, and $\|\xi\|=k$ becomes $\|\theta\|=kh$, yielding a circle (or sphere) of radius $r = kh$.
It gives rise to the following definition:

\begin{definition}[relative dispersion error]\label{def:disp_err}
Let $A$ be a continuous operator and let $A_h$ be its discretization.
Let $\phi$ be a propagation direction (in three dimensions $\phi = [\phi_{azimuthal},\phi_{polar}]$) in the Fourier space.
Let $r=kh$ and let $r_{num}$ be the radius for which the symbol $\widetilde{A}_h(\theta)$, restricted to propagation direction $\phi$, switches sign (assuming it happens only once).
The relative dispersion error is defined (see, e.g., \cite{jo1996optimal}) as
\begin{equation}\label{eq:disp_err}
e(\phi) = \frac{r}{r_{num}}-1.
\end{equation}
Notice that since the radius is proportional to the inverse of the wave speed, this is the difference of wave speeds, normalized.
\end{definition}

Note that the symbol of the discrete operator is generally not a radial function, so the dispersion error depends on the propagation direction.
Dispersion correction methods aim at reducing the error defined in Definition \ref{def:disp_err}, although a theoretical limitation is that this error cannot be completely eliminated when the dimension is higher than 1.
In the next subsection, we extend these definitions to vector equations, particularly, to the elastic Helmholtz equation in mixed formulation \eqref{eq:elasticMixed}.

\subsection{Numerical dispersion for saddle-point systems}\label{subsec:dispersion_saddle}

As defined in the previous subsection, the dispersion relation of a scalar operator is the kernel of its symbol.
In systems of equations, we have a matrix of symbols rather than a scalar symbol.
The dispersion relation in the vector case is the set of $\theta\in[-\pi,\pi]^d$ for which the vector plane wave $\bfv e^{i \bfx\cdot \theta}$ is a solution to the homogeneous system.
In the constant-coefficient scenario,
this is the set for which the determinant of the matrix of symbols is zero; see \cite{cocquet2024asymptotic}.
The following definition encapsulates it.

\begin{definition}[dispersion relation, vector equations]\label{def:disp_rel_vec}
Let $\mathcal A$ be a continuous or discrete operator representing a system of PDEs $\mathcal A \vec v = \vec f$.
Then its dispersion relation is
\begin{equation}\label{eq:disp_rel_vec}
\left\{ \theta\in[-\pi,\pi]^d \mid \det\left(\widetilde{\mathcal A}(\theta)\right) = 0 \right\}
\end{equation}
where $\widetilde{\mathcal A}(\theta)$ is the matrix of symbols of $\mathcal A$.
\end{definition}

For the elastic Helmholtz (vector) equation, since the equation admits two wavenumbers \eqref{eq:ContVelocities},
we split the discussion to shear and pressure dispersion relations.
In continuous space, the above-mentioned dispersion relations are circles (in two dimensions) or spheres (in three dimensions) of radii
\begin{equation}\label{eq:radii}
r_s = \omega\sqrt{\frac{\rho}{\mu}} = \frac{\omega}{V_s} \quad \text{and} \quad
r_p = \omega\sqrt{\frac{\rho}{\lambda + 2\mu}} = \frac{\omega}{V_p}.
\end{equation}

To analyze the discrete dispersion of the shear and pressure modes for the elastic Helmholtz equation in mixed formulation \eqref{eq:elasticMixed}, 
we first recall the (primal) Schur decomposition of generalized saddle-point systems; see \cite{benzi2005numerical}:
\begin{equation}\label{eq:schur_primal_decomp}
\mathcal{H} =
\begin{bmatrix}
A & B_1^T \\
B_2 & -C
\end{bmatrix}
 = \begin{bmatrix}
 I & -B_1^T C^{-1}\\
O & I
 \end{bmatrix}
 \begin{bmatrix}
 S_p & \\
  & -C
 \end{bmatrix}
 \begin{bmatrix}
 I & O \\
 -C^{-1}B_2 & I
 \end{bmatrix},
\end{equation}
where $S_p=A+B_1^T C^{-1} B_2$.
This decomposition holds for any discretization of \eqref{eq:elasticMixed}, but also for the matrix of symbols of both the discrete and the continuous operators.
In our case, the $C$ block of \eqref{eq:elasticMixed} is a diagonal matrix with $1/(\lambda+\mu)$ on its diagonal, so $\widetilde{C}$ is nonzero. 
When analyzing symbols, we assume spatial homogeneity, hence $\widetilde{C}$ does not depend on $\theta$. 
Hence, $\det(\mathcal{H}) = 0$ if and only if $\det(S_p)=0$.
We notice that $S_p$ is the elastic Helmholtz operator in grad-div formulation \eqref{eq:elasticHelmEquivalent}.
Thus, analyzing the dispersion relation of the mixed formulation \eqref{eq:elasticMixed} is equivalent to analyzing the dispersion relation of its primal Schur complement \eqref{eq:elasticHelmEquivalent}.

The discrete dispersion relation splits into shear and pressure curves (or surfaces), as in the continuous case.
However, unlike in the continuous case, neither is a perfect circle (or sphere).
To calculate the shear and pressure radii in the Fourier space, we first pick a propagation direction $\phi$.
Then we sample the symbol of the discrete operator over the ray of direction $\phi$ in $[-\pi,\pi]^d$.
Next, we apply eigendecomposition to the resulting matrix of symbols.
We identify the pressure mode as the eigenvalue whose eigenvector is closest to the propagation direction. 
The remaining eigenpairs correspond to the shear modes.
Then we observe that at some distance from the origin in $[-\pi,\pi]^d$ the pressure eigenvalue vanishes; we call this distance $r_p^{num}$, and similarly we define $r_s^{num}$.
Now, the definition of shear and pressure dispersion errors is directly generalizable from Definition \ref{def:disp_err}.

\subsection{Proof of absence of numerical dissipation}\label{subsec:no_dissipation}

Another form of numerical error can occur in the case of a vector equation with a non-symmetric discretization.
\emph{Numerical dissipation} manifests as artificial attenuation, represented by the imaginary part of the symbolic eigenvalues.
In scalar equations in the frequency domain, discretized by real symmetric stencils, there is evidently no numerical dissipation.
Nevertheless, the absence of numerical dissipation is not straightforward when a vector equation is discretized by a non-symmetric block structure: even when each block's stencil is symmetric, the matrix of symbols might have non-real eigenvalues.
This gives rise to the following definition:

\begin{definition}[numerical dissipation]\label{def:dissipation}
Let $\mathcal{H}_h,$ $\lambda_p$ and $\lambda_s$ be as in Definition \ref{def:polarization-s}.
Then the numerical dissipation, or the relative dissipation error, for the pressure and shear modes is given by
\begin{equation}\label{eq:dissipation}
e_p^{dissipation}(\theta) = \frac{|\text{Im}(\lambda_p(\theta))|}{|\lambda_p(\theta)|} ,
\quad \text{and} \quad
e_s^{dissipation}(\theta) = \frac{|\text{Im}(\lambda_s(\theta))|}{|\lambda_s(\theta)|}
\end{equation}
respectively.
\end{definition}

In the next theorem, we prove that in our case, there is no numerical dissipation.

\begin{theorem}\label{thm:no_dissipation}
Let $\mathcal{H}_h$ be a GMAC discretization  \eqref{eq:MAC-dispersion} of the elastic Helmholtz equation in mixed formulation \eqref{eq:elasticMixed} with constant Lam{\'e} coefficients $\lambda$ and $\mu$ and constant density $\rho$.
Then there is no numerical dissipation; that is, the eigenvalues of the matrix of symbols are real.
\end{theorem}

\begin{proof}
The shear and pressure eigenpairs of $\tilde{\mathcal{H}}(\theta)$ are determined by its primal Schur complement, $\widetilde{S}_p$; see Subsection \ref{subsec:dispersion_saddle}.
In a homogeneous medium, it takes the form
\begin{equation}\label{eq:schur-primal}
\widetilde{S}_p = \tilde{A} + \tilde{B}_1^T \tilde{C}^{-1}\tilde{B}_2
= \tilde{A} -(\lambda+\mu) \tilde{B}_1^T\tilde{B}_2
= -\mu\tilde{\vec\grad}_h^T\tilde{\vec\grad}_h^a +\rho \widetilde{M}^{b,c}-(\lambda+\mu) (\tilde{\div})_h^T (\tilde{\div})_h^a.
\end{equation}

First, we observe that for every $a,b,c\in\mathbb{R}$ the stencils for the (vector) Laplacian $\tilde{\vec\grad}_h^T\tilde{\vec\grad}_h^a$ and the mass $M^{b,c}$ are symmetric and hence the symbol of the leading block, $\tilde{A}(\theta)$, is real for every $\theta$.
Hence, it remains to prove that the eigenvalues of this grad-div operator are real.
Let $\lambda$ be an eigenvalue of $\tilde{B}_1^T\tilde{B}_2$.
Then either $\lambda \!=\! 0 $ or $\lambda$ is an eigenvalue of $\tilde{B}_2\tilde{B}_1^T$.
Indeed, let $\bfv$ be an eigenvector of $\tilde{B}_2\tilde{B}_1^T$ corresponding to an eigenvalue $\lambda\neq0$ and denote            $\bfw = \tilde{B}_1^T\bfv$.
Then $\tilde{B}_2\tilde{B}_1^T \bfw = \lambda \tilde{B}_2 \bfv = \lambda \bfw$.
We recall that the stencils for the first derivatives \eqref{eq:del_spread} were chosen such that the $a$-spread divergence applied to the standard gradient will yield the $a$-spread Laplacian of 9 points \eqref{eq:JSS_sten} in two dimensions or the Laplacian of 19 points \eqref{eq:Lap-19-pt-sten} in three dimensions.
Since $\tilde{B}_2\tilde{B}_1^T$ is the symbol of the scalar Laplacian discretized by a symmetric stencil, $\lambda\in\mathbb{R}$. This completes the proof that the GMAC discretization \eqref{eq:MAC-dispersion} yields no numerical dissipation.
\end{proof}


\subsection{Proof of absence of shear-mode separation}\label{subsec:no_shear_split}

Another form of error applies only to vector equations in three or more dimensions. 
In the elastic case in three dimensions, there are two shear eigenvalues, and \emph{shear mode separation} is present when they are not equal. The following definition refers to the existence of two distinct shear modes. We prove in this subsection that this case does not happen in our discretization.

\begin{definition}[shear mode separation]\label{def:shear_split}
Let $\mathcal{H}_h$ be a discretization of the left-hand side from \eqref{eq:elasticHelmEquivalent} in three dimensions.
Let $\lambda_p$ be defined as in Definition \ref{def:polarization-p} and let $\lambda_{s_1}$ and $\lambda_{s_2}$ be the two remaining eigenvalues.
If $\lambda_{s_1}\neq\lambda_{s_2}$ we say there is shear-split of size $\lambda_{s_2}-\lambda_{s_1}$.
\end{definition}

In the next theorem, we prove that in our case, there is no shear mode separation.

\begin{theorem}\label{thm:no_shear_split}
Let $\mathcal{H}$ be a GMAC discretization \eqref{eq:MAC-dispersion} of the elastic Helmholtz equation in mixed formulation \eqref{eq:elasticMixed} in three dimensions with constant Lam{\'e} coefficients $\lambda$ and $\mu$, and constant density $\rho$.
Then there is no shear mode separation; that is, for every $\theta$, the two eigenvalues of the $3\times3$ matrix of symbols $\widetilde{S}_p(\theta)$ from \eqref{eq:schur-primal} corresponding to the shear wave velocity are equal.
\end{theorem}

\begin{proof}
As shown in the proof of Theorem \ref{thm:no_dissipation}, for every $\theta$, if $\nu$ is an eigenvalue of $\tilde{B}_1^T(\theta)\tilde{B}_2(\theta)$ from \eqref{eq:schur-primal}, then either $\nu \!=\! 0$ or $\nu$ is an eigenvalue of $\tilde{B}_2(\theta)\tilde{B}_1^T(\theta)$.
Since $\tilde{B}_1^T(\theta)\tilde{B}_2(\theta)$ is a rank one matrix, the eigenvalue 0 has multiplicity 2.
Recall that $\widetilde{A}(\theta)$ is a scalar matrix with eigenvalue $\mu\widetilde{\Delta}_h(\theta)+\rho\omega^2\widetilde{M}^{b,c}(\theta)$ that corresponds to the shear mode.
Hence, for the matrix $\widetilde{S}_p(\theta)$, the eigenvalue $\nu$, corresponding to the shear mode, has multiplicity 2.
\end{proof}

\subsection{Definitions of polarization error}\label{subsec:polarization}

Another form of error that can appear in vector equations is \emph{polarization error}, which measures how much an eigenvector corresponding to a certain mode deviates from its physical direction.
The elastic Helmholtz equation describes pressure waves, which, in the continuous case, consist of particle motion aligned with the propagation direction, and shear waves in which the motion is perpendicular to the propagation direction.
In the discrete case, however, the pressure waves might deviate from the propagation direction and the shear waves might deviate from the orthogonal direction.
This gives rise to the following definition, which appears, e.g., in \cite{seriani2008dispersion}. 
\begin{definition}[Pressure polarization error, \cite{seriani2008dispersion}]\label{def:polarization-p}
Let $\mathcal{H}_h$ be a discretization of \eqref{eq:elasticHelmEquivalent} in two dimensions.
Let $\theta \in [-\pi,\pi]^d$ and denote the propagation direction corresponding to $\theta$ by $\hat\bfd = \frac{\theta} {\|\theta\|}$.
Let $\hat\bfv_p$ be the normalized eigenvector of the matrix of symbols $\widetilde{\mathcal A}_h(\theta)$ closest in angle to $\hat\bfd$.
Then the pressure polarization error in $\theta$ is
\begin{equation}\label{eq:polarization_pressure}
e^{polarization}_p(\theta) = \arccos\left(|\hat\bfd \cdot \hat\bfv_p|\right) = \arcsin\left( \|(I-\hat\bfd \hat\bfd^T)\hat\bfv_p\| \right),
\end{equation}
where $I$ is the $2\times2$ identity matrix.
\end{definition}
In \cite{seriani2008dispersion},  polarization (or polarity) error was defined in the context of a spectral element method, and the scheme was assumed to be symmetric.
As a result of the symmetry, the discrete shear and pressure modes were orthogonal, and hence their deviation from the corresponding physical directions was equal.
For this reason, it was enough to define polarization error for pressure waves only.
In our case, our scheme is non-symmetric, so the two senses of polarization may differ.
We therefore suggest the following similar definition.
\begin{definition}[Shear polarization error]\label{def:polarization-s}
Under the notations of Definition \ref{def:polarization-p}, 
let
$\hat\bfv_s$ be the remaining normalized eigenvector of $\tilde{\mathcal{H}}_h(\theta)$. 
Then the shear polarization error in $\theta$ is
\begin{equation}\label{eq:polarization_shear}
e^{polarization}_s(\theta) = \arcsin\left( |\hat\bfd \cdot \hat\bfv_s|\right).
\end{equation}
\end{definition}

We gave the last definition in two dimensions for simplicity.
In three dimensions, there can theoretically be two different shear polarization errors.
In the next section, we use the definitions above to generate numerical examples, typically by maximizing the error measures over $\theta\in[-\pi,\pi]^2$ in two dimensions or over $\theta\in[-\pi,\pi]^3$ in three dimensions.
Even when the maximum dispersion error is low, the maximum polarization error might still be high, indicating that the discrete scheme captures the phases but propagates them in the wrong directions.

\section{Dispersion examples}\label{sec:numerical}

In this section, we demonstrate the properties of our GMAC discretization \eqref{eq:MAC-dispersion} and our EDC process of Section \ref{sec:lame},
by sampling the symbol of the discrete system, derived in Section \ref{sec:analysis}.
In Subsection \ref{subsec:numerical2D}  we analyze the dispersion error as explained in Subsection \ref{subsec:dispersion_saddle} and compare it to existing dispersion correction methods for the elastic Helmholtz equation in two dimensions.
In Subsection \ref{subsec:numerical3D} we generate similar results in three dimensions, and also refer to the polarization error, defined in Subsection \ref{subsec:polarization}.

\subsection{Two-dimensional examples}\label{subsec:numerical2D}

In Fig. \ref{fig:disp_rel_HO_JSS} we display the pressure and shear dispersion relations for the GMAC discretization \eqref{eq:MAC-dispersion} in two dimensions, with and without EDC. We use two choices of weights,
\begin{itemize}
\item \textbf{GMAC-ST:} Using \eqref{eq:MAC-dispersion} with weights $[a,b,c] = [2/3,2/3,1/3]$, which is our elastic discretization combined with the weights of the 4th-order discretization of the acoustic Helmholtz equation suggested by Singer and Turkel in \cite{singer1998high}.
\item \textbf{GMAC-JSS:} Using \eqref{eq:MAC-dispersion} with the weights $[a,b,c] = [0.5461,0.6248,0.3752]$, which is our elastic discretization combined with the weights of the acoustic dispersion correction discretization suggested by Jo, Shin, and Suh in \cite{jo1996optimal}.
\end{itemize}
We use $\mu=\rho=1$ and $\lambda=10$, leading to a Poisson ratio of about $\sigma = 0.45$, and $G_s=4$ or $2.5$ grid points per shear wavelength.
The high Poisson ratio means pressure waves have significantly more grid points per pressure wavelength, $G_p = 13.84$ or $G_p = 8.65$, at which the phases are already well aligned. 
Therefore, we use $t_p=1$ in all of the cases (with the corresponding correction for $\lambda$), and for $t_s$, we choose a value that nearly minimizes the dispersion error over all propagation directions.

\begin{figure}
\begin{center}	\newcommand{\image}[1]{\includegraphics[width=0.32\linewidth]{./#1}}
    \subfigure[\parbox{3.5cm}{\centering \footnotesize GMAC-ST, no EDC \\ $G_s=4$}]{\image{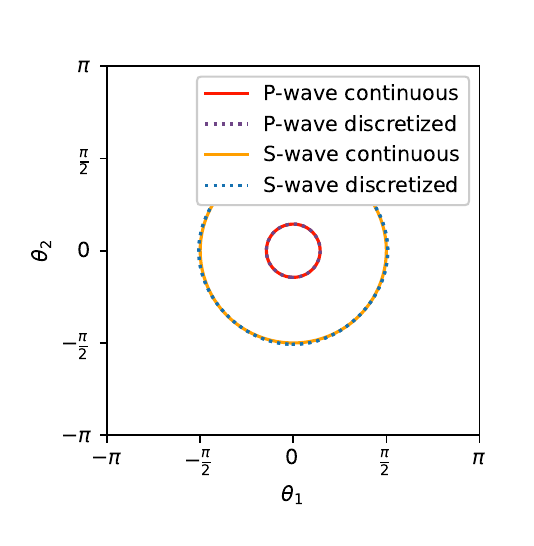}\label{fig:disp_rel_4g_HO}}
    \subfigure[\parbox{3.5cm}{\centering \footnotesize GMAC-ST, with EDC\\ $t_s=0.993$, $G_s=4$}]{\image{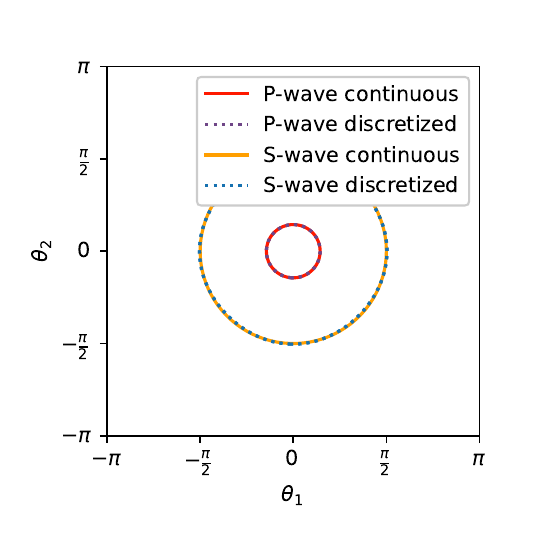}\label{fig:disp_rel_4g_HO_kstar}}
    \subfigure[\parbox{3.5cm}{\centering \footnotesize
GMAC-JSS, no EDC \\ $G_s=4$}]{\image{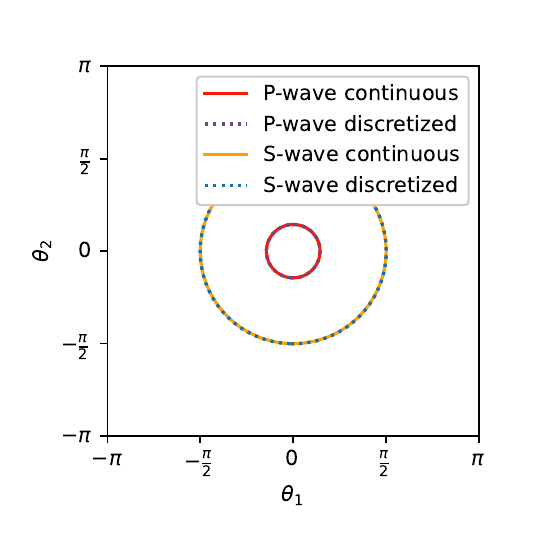}\label{fig:disp_rel_4g_JSS}}
    \subfigure[\parbox{3.5cm}{\centering \footnotesize GMAC-ST, no EDC \\ $G_s=2.5$}]{\image{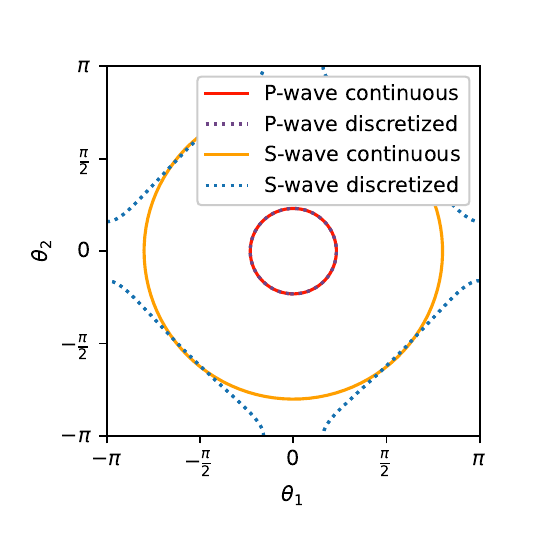}\label{fig:disp_rel_2p5g_HO}}
    \subfigure[\parbox{3.5cm}{\centering \footnotesize GMAC-ST, with EDC \\ $t_s=0.934$, $G_s=2.5$}]{\image{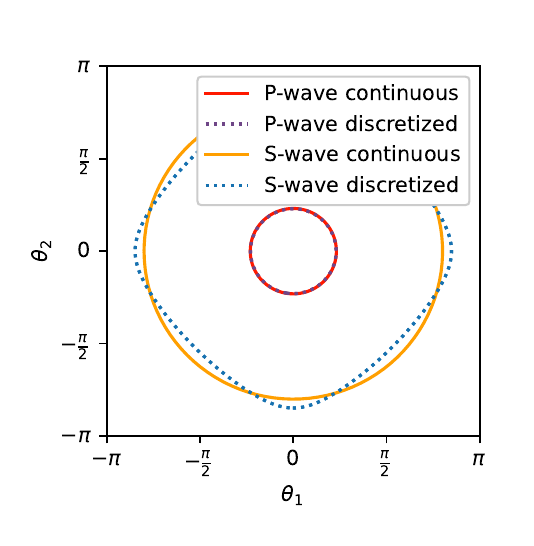}\label{fig:disp_rel_2p5g_HO_kstar}}
    \subfigure[\parbox{3.5cm}{\centering \footnotesize GMAC-JSS, no EDC \\ $G_s=2.5$}]{\image{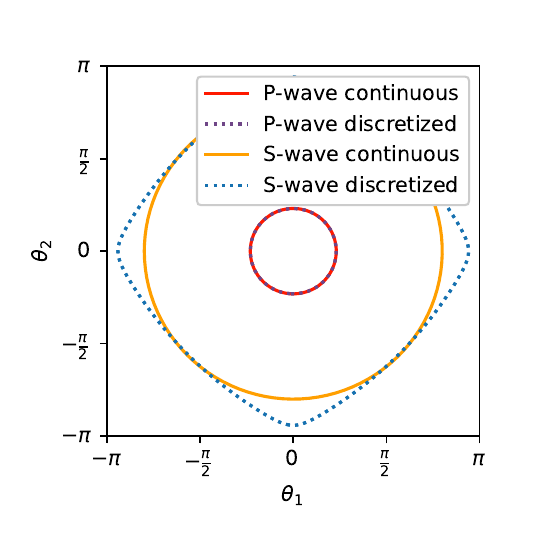}\label{fig:disp_rel_2p5g_JSS}}\\
\end{center}
\caption{Dispersion relations for the elastic Helmholtz equation discretized by different GMAC weights, with or without EDC, at Poisson ratio $\sigma = 0.45$.
}\label{fig:disp_rel_HO_JSS}
\end{figure}

We observe that for $G_s=4$ there is no significant difference between the methods, see Figs. \ref{fig:disp_rel_4g_HO}, \ref{fig:disp_rel_4g_HO_kstar} and \ref{fig:disp_rel_4g_JSS}.
We note that all of these methods have better alignment of phases compared to standard MAC, which  is not presented here, and hence a choice of weights for the GMAC is enough to correct the dispersion in this case.
However, for $G_s=2.5$ the EDC process becomes meaningful:
in Fig. \ref{fig:disp_rel_2p5g_HO} the shear dispersion relation corresponding to the GMAC-ST discretization method has four connected components.
The GMAC-JSS discretization method yields a better dispersion relation; see Fig. \ref{fig:disp_rel_2p5g_JSS}.
EDC can further reduce the dispersion error, as seen in Fig. \ref{fig:disp_rel_2p5g_HO_kstar}.

\begin{figure}
\begin{center}
	\newcommand{\image}[1]{\includegraphics[width=0.32\linewidth]{./#1}}
    \subfigure[\footnotesize Without pressure correction]{\image{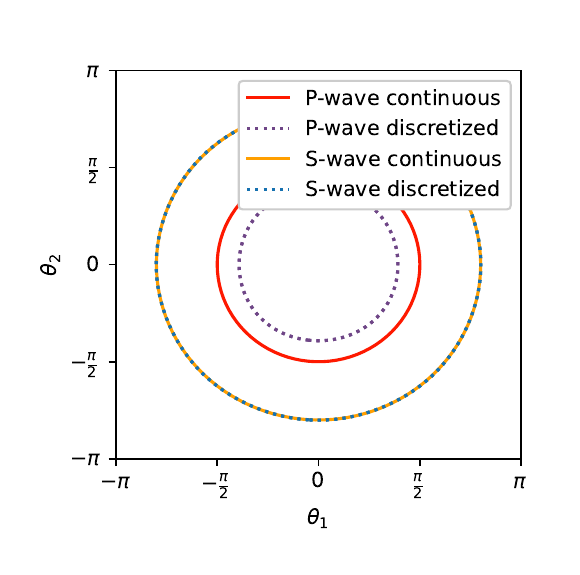}\label{fig:disp_rel_2p5g_CGX_no_pressure_corr}}
    \subfigure[\footnotesize Assuming $t_p=1$ in \eqref{eq:s_def}]{\image{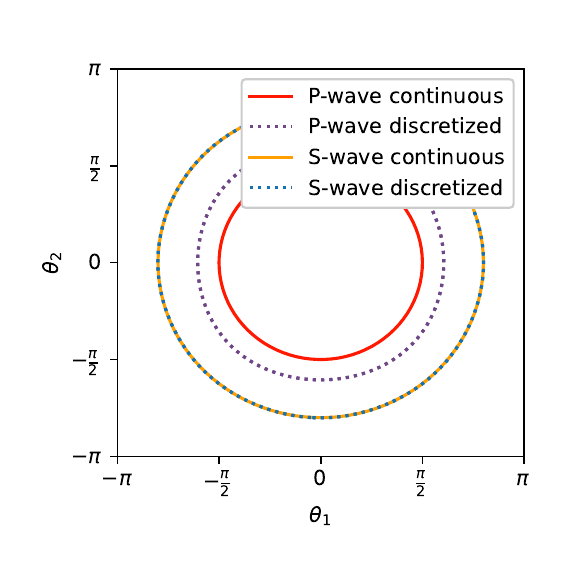}\label{fig:disp_rel_2p5g_CGX_basic_pressure_corr}}
   \subfigure[\footnotesize With $t_p=0.72186 $ in \eqref{eq:s_def}]{\image{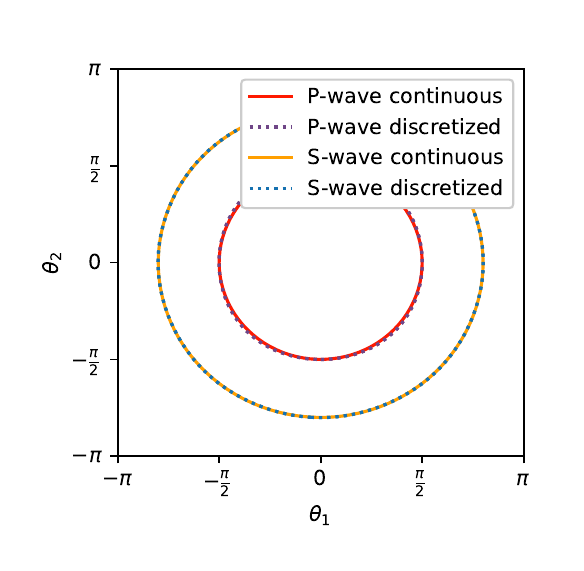}\label{fig:disp_rel_2p5g_CGX}}\\
\end{center}
\caption{Dispersion relations for shear and pressure waves, with  $G_s=2.5$ and $G_p=4$ (Poisson ratio $\sigma = 0.18$), achieved by partial or full application of our EDC process: 
on the left, just the first step is applied with $a,b,c$ and $t_s$ as reported in \cite{cocquet2017finite} for acoustic Helmholtz with $G=G_s$.
In the center, applying the full EDC process, assuming $t_p=1$.
On the right, applying the full process where $t_p$ is the optimal parameter reported in \cite{cocquet2017finite} for $G=G_p$.
}\label{fig:disp_rel_CGX}
\end{figure}

The following experiment, depicted in Fig. \ref{fig:disp_rel_CGX}, emphasizes the need for all of the steps of the EDC process from Section \ref{sec:lame}, and the necessity of choosing $t_p\neq 1$ in \eqref{eq:s_def} in some scenarios.
In \cite{cocquet2017finite}, the weights $[a,b,c]$ and a real shifting factor $t$ were optimized in the following sense: 
The acoustic Helmholtz equation with wavenumber $k^*=tk$, discretized by the stencil \eqref{eq:JSS_sten} with $[a,b,c]$, has the best phase alignment with the continuous equation with wavenumber $k$, summing the error over propagation directions.
We re-performed the optimization for each number of grid points per wavelength $G$, and report some optimal values there.
Note that the resulting weights $[a,b,c]$ without a shifted wavenumber do not necessarily form a valid approximation of the equation, as the shifting parameter and the stencil weights were optimized together.
To adapt this method to the elastic Helmholtz equation, we use our GMAC discretization \eqref{eq:MAC-dispersion} for some $[a,b,c]$ values reported in \cite{cocquet2017finite}, and apply EDC with $t_p$, the reported value for the acoustic case with $G=G_p$, and $t_s$, the reported value for $G=G_s$.

In Fig. \ref{fig:disp_rel_CGX} we take $G_s=2.5$ and Poisson ratio $\sigma=0.18$, so that $G_p=4$.
The values reported in \cite{cocquet2017finite} for $G=2.5$ are
$[a_1,b_1,c_1] = [0.5324,1.0553,0.0003]$ and $t_1=0.72186$ and for $G=4$,
the reported optimal parameters are $[a_2,b_2,c_2] = [0.6054,1.0532,0.0002]$ and $t_2 = 0.87725$.
Since, in the elastic case, the shear and pressure waves are coupled into a single (vector) equation, we have to choose one set of stencil parameters.
We therefore take $[a_1,b_1,c_1]$ which are optimal for $G_s$.
We set $t_s = 0.72186$ and $t_p = 0.87725$, where the latter is nearly optimal for the pressure correction since $a_1,b_1,c_1$ are relatively close to $a_2,b_2,c_2$.
We display in Fig. \ref{fig:disp_rel_CGX} the shear and pressure dispersion relations for three different partial versions of EDC:
\begin{itemize}
\item In Fig.  \ref{fig:disp_rel_2p5g_CGX_no_pressure_corr}, only the first step of the EDC process is performed: $\omega$ is replaced by $\omega^* = t_s\omega$, with $t_s = 0.72186$, but no special treatment addresses the pressure waves. 
We observe that the shear dispersion relation is very well aligned with the continuous one, while the pressure waves are far from alignment.
\item In Fig. \ref{fig:disp_rel_2p5g_CGX_basic_pressure_corr}, we apply the full EDC process, with $t_s = 0.72186$ and $t_p=1$ in \eqref{eq:s_def}. 
We observe that the shear waves remain well corrected, while the pressure-wave correction overshoots the target, yielding an error in the opposite direction. 
This results from the fact that the stencil weights $a_1$, $b_1$, and $c_1$, without any shifted wavenumber, do not form a valid approximation (and setting $t_p=1$ effectively cancels the shift in the pressure wavenumber).
\item In Fig. \ref{fig:disp_rel_2p5g_CGX} we depict the full EDC process, with the reported values $t_s = 0.72186$ and $t_p=0.87725$.
We observe nearly perfect alignment of both shear and pressure dispersion relations.
\end{itemize}
We note that for some applications, a method requiring re-optimizing multiple weights for every configuration might be too costly. 
However, the final experiment in Fig. \ref{fig:disp_rel_2p5g_CGX} demonstrates the ability of our EDC process to correct the dispersion of both shear and pressure waves in the elastic case, in the same efficiency as known methods do for the acoustic Helmholtz equation.

In Fig. \ref{fig:disp_err} we present curves of the ratio between discrete and continuous wave speeds $c_{num}/c_{cont}$ vs. $1/G_s$ for different propagation directions.
The density is $\rho=1$ and the Lam{\'e} coefficients are $\lambda=10$ and $\mu=1$, corresponding to Poisson ratio $\sigma = 0.4545$.
In Fig. \ref{fig:disp_err_vs_OneOverG_fourth_2D} we examine the GMAC-ST discretization without EDC.
In Fig. \ref{fig:disp_err_vs_OneOverG_Opt_2D} we apply the same discretization with EDC.
For each $1/G_s$ value we re-optimize $t_s$, and take $t_p=1$ in \eqref{eq:s_def}.
Because of the large Poisson ratio, the pressure requires no further correction.
For the shear waves, GMAC-ST with EDC is significantly more accurate than without it, yielding a $6\%$ error with as few as $2.5$ grid points per shear wavelength; see Fig. \ref{fig:disp_err_vs_OneOverG_Opt_2D}. This is compared to $>20\%$ in Fig. \ref{fig:disp_err_vs_OneOverG_fourth_2D}.
We repeat the two experiments for the GMAC-JSS discretization.
For $G_s=2.5$, we have nearly $15\%$ error without EDC, see Fig. \ref{fig:disp_err_vs_OneOverG_JSS}, compared to only $5\%$ error with EDC, see Fig. \ref{fig:disp_err_vs_OneOverG_JSS_Opt}.

\begin{figure}
\begin{center}	\newcommand{\image}[1]{\includegraphics[width=0.48\linewidth]{./#1}}
    \subfigure[\footnotesize GMAC-ST, no EDC]{\image{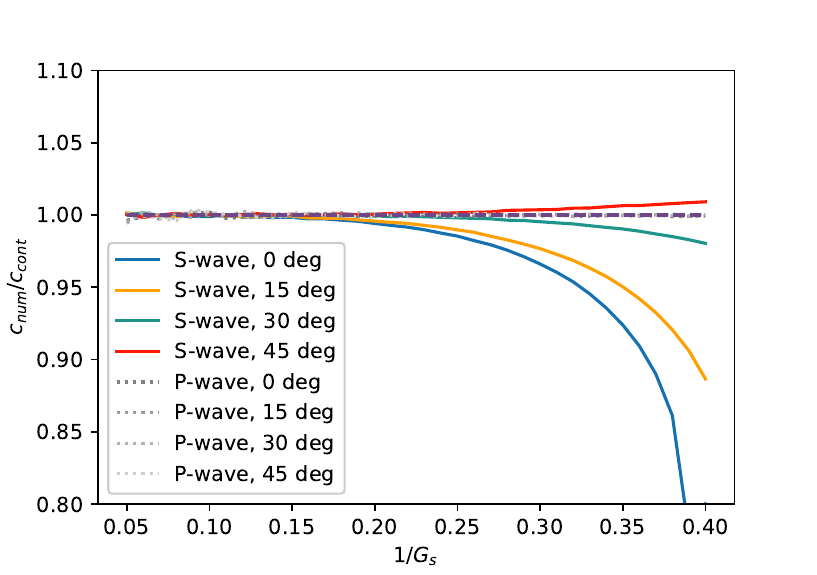}\label{fig:disp_err_vs_OneOverG_fourth_2D}}
   \subfigure[\footnotesize GMAC-ST, with EDC]{\image{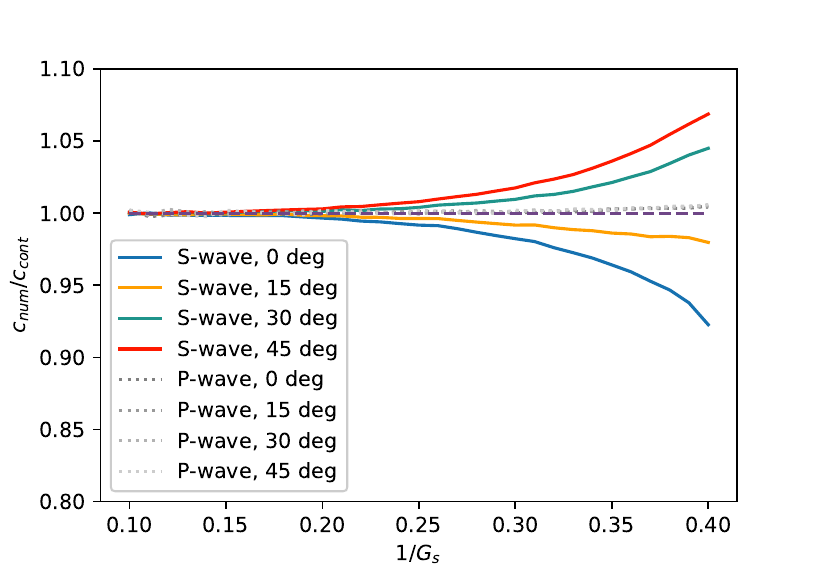}\label{fig:disp_err_vs_OneOverG_Opt_2D}}
   \subfigure[\footnotesize GMAC-JSS, no EDC]{\image{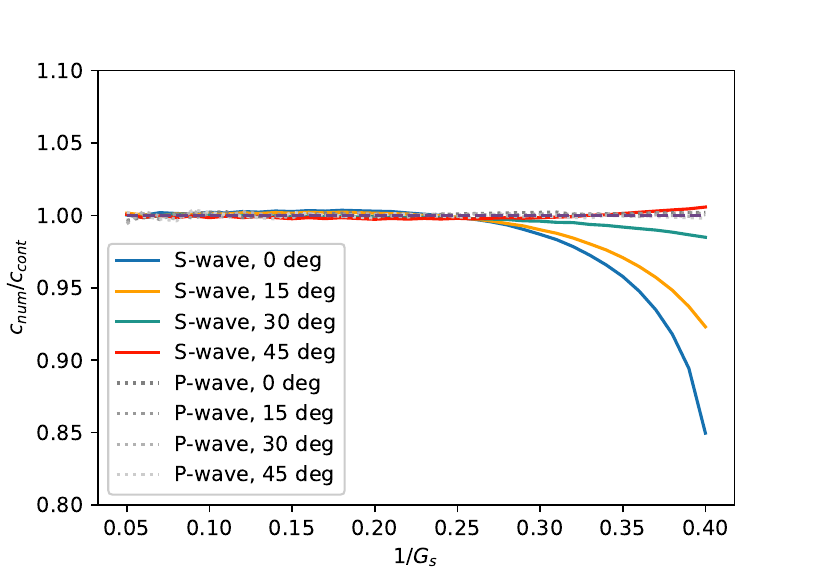}\label{fig:disp_err_vs_OneOverG_JSS}}
   \subfigure[\footnotesize GMAC-JSS, with EDC]{\image{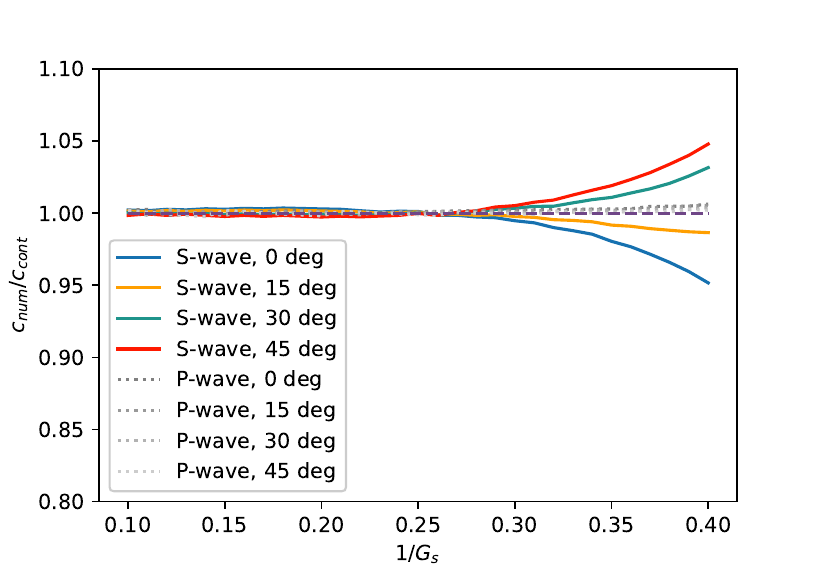}\label{fig:disp_err_vs_OneOverG_JSS_Opt}}
   \\
\end{center}
\caption{Phase velocity ratio vs. $1/G_s$ for different propagation directions for the elastic Helmholtz equation in two dimensions, discretized by GMAC-ST and GMAC-JSS.
}\label{fig:disp_err}
\end{figure}

For all the above-mentioned GMAC versions, we also measure numerical dissipation; see Definition \ref{def:dissipation}.
The maximal normalized numerical dissipation did not exceed $~8.5\cdot 10^{-17}$.
This value is numerically zero and reflects our implementation of symbols as complex exponents.
This aligns well with Theorem \ref{thm:no_dissipation},
which claims that our family of discretization schemes produces no numerical dissipation.

\paragraph*{Comparison to other dispersion correction methods}
By quantitatively measuring the results in Fig. \ref{fig:disp_err_vs_OneOverG_JSS_Opt}, we observe that only 
$G_s=3.27$ grid points per shear wavelength are needed to keep the maximal phase error below $1\%.$ 
For comparison, in \cite{gu201321}, the authors compare compact and non-compact FD schemes for elastic wave simulation and report that their compact 9-point scheme
requires $G_s=9$, their non-compact 25-point scheme
requires as few as $G_s=2.7$, and their 21-point scheme needs $G_s=3.3$ grid points per shear wavelength to
keep the dispersion error below $1\%$.
We also compare to the spectral element method of \cite{seriani2008dispersion}, which reports an error below 1\% for $G_s=4$.
That is, our GMAC-JSS with EDC forms a compact method that requires \emph{$\times2.75$ fewer grid points per shear wavelength} compared to an existing compact method, and is even \emph{competitive with non-compact methods}.

\subsection{Three dimensional examples}\label{subsec:numerical3D}

In this subsection, we present numerical experiments in three dimensions.
We use the following discretizations:
\begin{itemize}
\item \textbf{GMAC (ST analog):} Using \eqref{eq:MAC-dispersion} with the weights $[a,b,c] = [1/3,1/2,1/2]$, which is our elastic discretization applied with the weights of the acoustic fourth-order discretization from Subsection \ref{subsec:acoustic3D}.
\item \textbf{GMAC (JSS analog):} Using \eqref{eq:MAC-dispersion} with the weights $[a,b,c] = [0.06,0.37,0.63]$, which is our elastic discretization applied with the weights optimized to reduce numerical dispersion over propagation directions for $G_s>2.5$.
For each $G_s$ value, $a$ and $b$ were taken in the range $[0,1]$ with jumps of $0.01$, and the dispersion error \eqref{eq:disp_err} was calculated for three propagation directions: axis direction, face diagonal, and body diagonal.
To compute the radii in \eqref{eq:disp_err}, the symbolic shear and pressure eigenvalues from Section \ref{sec:analysis} were sampled along a ray in the given direction, with jumps of 0.001.
Setting $c=1-b$ (which is the case in the ST method and in the two-dimensional JSS method), the weights $a$ and $b$ were chosen to minimize the maximal error.
\end{itemize}

In Fig. \ref{fig:ball}, we depict the three-dimensional dispersion relation of the GMAC (ST analog), and observe a pattern similar to the two-dimensional dispersion relations of Fig. \ref{fig:disp_rel_2p5g_HO}.
In Fig. \ref{fig:max_polarization_vs_a} we show the maximal shear polarization error (see Definition \ref{def:polarization-s}) over all propagation directions, for different choices of weight $a$ for the first derivatives \eqref{eq:dxSpread} in two dimensions and \eqref{eq:del_spread} in three dimensions.
We report results for $G_s=2.5$, $[b,c] = [2/3,1/3]$ in two dimensions and $[b,c] = [1/2,1/2]$ in three dimensions.
Nevertheless, we observe that the polarization error does not depend on either $ G_s$ or the mass weights $b,c$.
The shear polarization error varies strongly with $a$.
For the smallest weight examined, $a=0.05$, the shear polarization error is about $25^\circ$ in two dimensions and $15^\circ$ in three dimensions, which makes such discretizations impractical. Interestingly, the optimal shear polarization error, about $3.2^\circ$ in two dimensions and $4.1^\circ$ in three dimensions, occurs near the weights corresponding to the high-order scalar Laplacian, $a=2/3$ in two dimensions and $a=1/3$ in three dimensions.
Hence, although these weights do not yield fourth-order accuracy for the elastic Helmholtz equation (as we show in Section \ref{sec:OOA}), they are still worth using to obtain a physically faithful discretization whose shear modes are nearly transversal to the propagation direction

The maximal pressure polarization error (see definition \ref{def:polarization-p}), which we do not present visually, is constant and is about $8.8^\circ$ in two dimensions and $10.2^\circ$ in three dimensions.
It is worth mentioning that the non-symmetry of the GMAC discretization \eqref{eq:MAC-dispersion} is an advantage here: 
The weight $a=1$, corresponding to a symmetric standard MAC, guarantees equal shear and pressure polarization errors.
Allowing the non-symmetry enables significantly smaller shear polarization errors.

\begin{figure}
\begin{center}
	\newcommand{\image}[1]{\includegraphics[width=0.4\linewidth]{./#1}}
  \subfigure[\footnotesize 3D dispersion relation for $G_s=2.6$]{\image{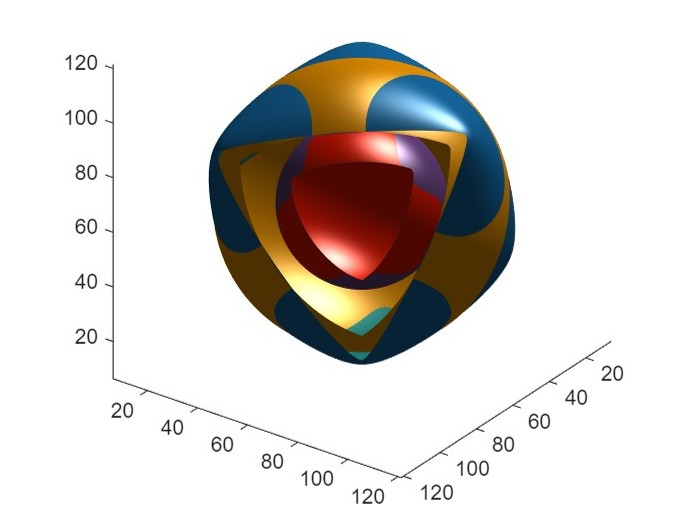}\label{fig:ball}}
     \hspace{20pt}
    \subfigure[\footnotesize Polarization vs. Laplacian weight]{\image{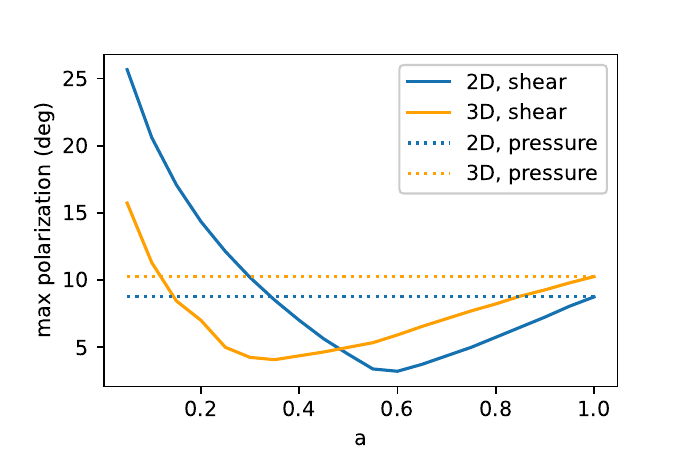}\label{fig:max_polarization_vs_a}}
\\
\end{center}
\caption{On the left:
shear and pressure three-dimensional dispersion relations for GMAC (ST analog) with $G=2.6$, for a sample of $126^3$ points $\theta\in[-\pi,\pi]^3$.
Orange: continuous shear, turquoise: discrete shear, red: continuous pressure, and purple: discrete pressure dispersion relations.
On the right: maximal polarization error vs. the Laplacian weight $a$ in the two-dimensional stencil \eqref{eq:JSS_sten} or the three-dimensional stencil  \eqref{eq:Lap-19-pt-sten}.
}\label{fig:ball_polarization}
\end{figure}

In two dimensions, the polarization was not a limitation, since all the discretizations in Subsection \ref{subsec:numerical2D} used $a=1,2/3$ or $0.5461$, which all guarantee reasonable polarization errors by Fig. \ref{fig:max_polarization_vs_a}. 
On the contrary, in three dimensions, the polarization presents significant limitations:
The near-zero $a$ in GMAC (JSS analog) yields a large shear-polarization error.

\begin{figure}
\begin{center}
	\newcommand{\image}[1]{\includegraphics[width=0.48\linewidth]{./#1}}
    \subfigure[\footnotesize 3D GMAC (ST analog), no EDC]{\image{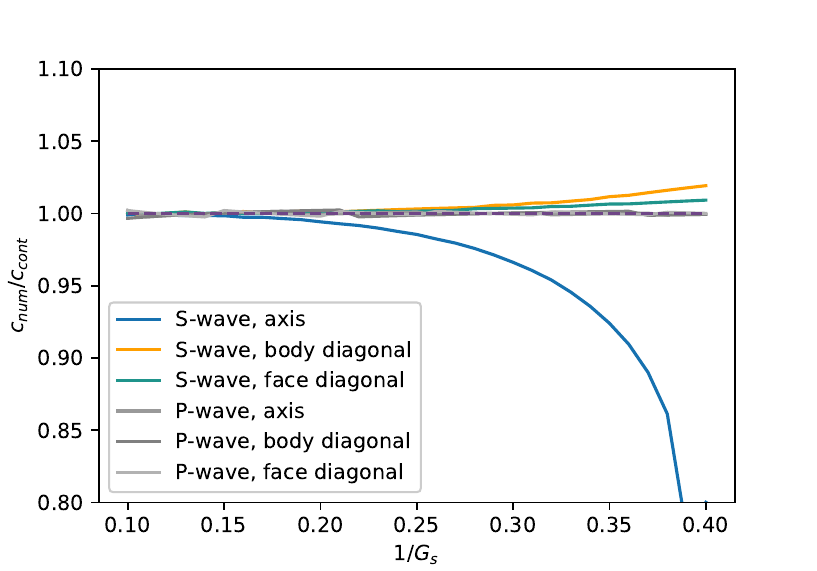}\label{fig:disp_err_vs_OneOverG_HO_3D}}
   \subfigure[\footnotesize 3D GMAC (JSS analog) with EDC, high polarization]{\image{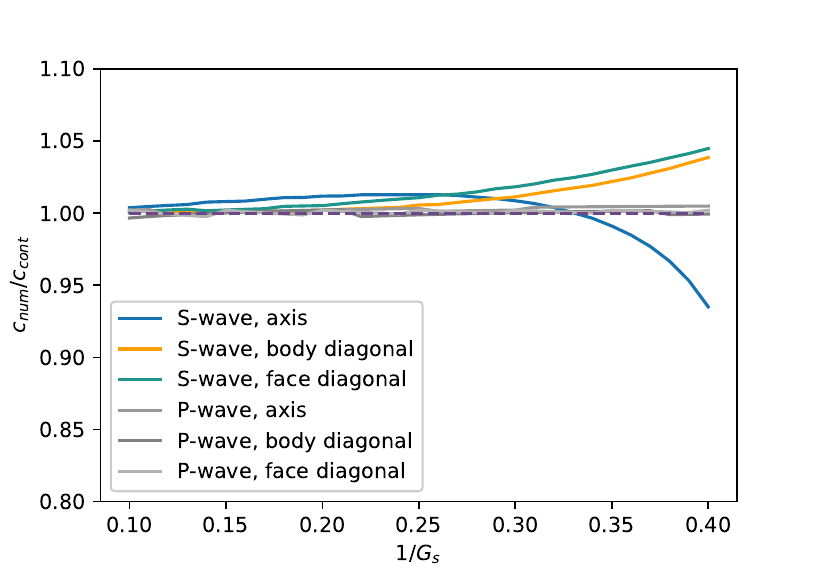}\label{fig:disp_err_vs_OneOverG_Corr_3D}}\\
\end{center}
\caption{Dispersion error vs. $1/G_s$ for and different propagation directions for the elastic Helmholtz equation in three dimensions.
On the left, for GMAC (ST analog) without EDC, and on the right, for GMAC (JSS analog) with EDC.
The latter has significant polarization error, see Fig. \ref{fig:max_polarization_vs_a}.
}\label{fig:disp_err_3D}
\end{figure}

In Fig. \ref{fig:disp_err_3D}, we display the phase-velocity ratio for several propagation directions in three dimensions vs. $1/G_s$.
In Fig. \ref{fig:disp_err_vs_OneOverG_HO_3D}
we show results for GMAC (ST analog) without EDC,
and in Fig. \ref{fig:disp_err_vs_OneOverG_Corr_3D} for GMAC (JSS analog) with EDC. 
The maximal dispersion error for the former is 22\% and for the latter 6.5\%, for $ G_S> 2.5$. 
We revisit the GMAC (JSS analog) discretization in Section \ref{sec:OOA} and show that it yields significantly higher truncation errors despite lower numerical dispersion, due to its large polarization error.
From the comparison of Fig. \ref{fig:disp_err_vs_OneOverG_HO_3D} to Fig. \ref{fig:disp_err_vs_OneOverG_fourth_2D}, 
we observe that the  GMAC (ST analog) behaves similarly to the two-dimensional GMAC-ST.
This is also true when adding a further correction through EDC, and for brevity, we do not include a visual demonstration for the three-dimensional case.

As expected from Theorem \ref{thm:no_dissipation}, in all the examined cases, the numerical dissipation (see Definition \ref{def:dissipation}) was numerically zero, similarly to the two-dimensional case. 
Moreover, the shear mode separation (see Definition \ref{def:shear_split}) was identically zero, as expected from Theorem \ref{thm:no_shear_split}.

\paragraph*{Comparison to other dispersion correction methods}
We quantitatively measure the errors presented in Fig. \ref{fig:disp_err_vs_OneOverG_HO_3D}, and achieve errors of about $6\%$ for $1/G_s = 0.4$ and $2\%$ for $ 1/G_s = 0.3$, for GMAC (ST analog) without EDC.
Our method achieves a dispersion error of only $0.58\%$  for $1/G_s = 0.2$.
For comparison, the compact FD method suggested in \cite{gosselin20143d} achieves a maximum phase error of $4\%$ for $1/G_s = 0.2$, as the authors report there.
That is, our method achieves a \emph{reduction of about $\times 7$ in dispersion error} for 5 grid points per shear wavelength compared to the compact method in \cite{gosselin20143d}.
We also compare our method to the non-compact method of \cite{li2016Fourth}, which uses grids with up to 3 cells distant from the stencil center in each direction.
Their method requires $G_s=3.7$ to keep the phase error below 1\%, as the authors report in \cite{li2016Fourth}.
In our GMAC-ST with EDC, a similar number of grid points per shear wavelength, $G_s=3.85,$ suffices to keep the error below 1\%. 
Overall, our method outperforms the compact method and is comparable to non-compact methods.

\section{Order of accuracy}\label{sec:OOA}

In this section, we first prove that our GMAC discretization method \eqref{eq:MAC-dispersion} has second-order accuracy for all weights $[a,b,c]$ in \eqref{eq:MAC-dispersion}.
While the acoustic Helmholtz equation achieves fourth-order accuracy for some compact stencils see Section \ref{sec:acoustic}, the GMAC with those stencil weights is only second-order accurate.
As we show below, this loss of accuracy is due to the cross-derivatives.
Second, we numerically demonstrate the order of accuracy and the magnitude of the truncation error for different GMAC methods, with and without EDC, in two and three dimensions.

\subsection{Order of accuracy analysis}\label{subsec:OOA_analysis}

We compute the order of accuracy in two dimensions; the generalization to three dimensions is straightforward.
We focus on the GMAC-ST; see the beginning of Subsection \ref{subsec:numerical2D}, since this is our analog of the fourth-order discretization for the acoustic Helmholtz equation.
Since the grad-div formulation \eqref{eq:elasticHelmEquivalent} and the mixed formulations \eqref{eq:elasticMixed} of the elastic Helmholtz equation are equivalent, one can discretize the grad-div formulation \eqref{eq:elasticHelmEquivalent} using GMAC by taking the primal Schur complement of \eqref{eq:MAC-dispersion}.
Therefore, it is enough to show that the resulting grad-div is second-order accurate.
To do so, we first show that the first derivatives are second-order accurate; then, building on that, we develop expressions for the cross-derivatives.

We shall show that the corresponding first derivative \eqref{eq:dxSpread} with $a=2/3$ is second-order accurate.
A direct expansion shows that the following  ``skew'' first derivative, corresponding to $a=-1/3$ in \eqref{eq:dxSpread}, is fourth-order accurate:
\begin{align*}
\textstyle
(\partial_{x_1})_{h/2}^{-1/3}f &\textstyle  = \frac{1}{2h}\left(f\left(x_1+\frac{h}{2},x_2+h\right) - f\left(x_1-\frac{h}{2},x_2+h\right) + f\left(x_1+\frac{h}{2},x_2-h\right) - f\left(x_1-\frac{h}{2},x_2-h\right)\right) \\
\textstyle & \textstyle = \partial_{x_1}f(x_1,x_2) + \mathcal{O}(h^4).
\end{align*}
We recall that the standard central-difference derivative of $h/2$ is second-order accurate:
\begin{equation*}
\textstyle
(\partial_{x_1})_{h/2}f=
\frac{1}{h}\left(f\left(x_1+\frac{h}{2},x_2\right) - f\left(x_1-\frac{h}{2},x_2\right)\right)= \partial_{x_1}f(x_1,x_2) + \mathcal{O}(h^2).
\end{equation*}
The first derivative \eqref{eq:dxSpread} with $a=2/3$ satisfies
\begin{equation*}
(\partial_{x_1})_{h/2}^{2/3} = \frac{5}{6}(\partial_{x_1})_{h/2} + \frac{1}{6}(\partial_{x_1})_{h/2}^{-1/3}
\end{equation*}
and this term necessarily has an error of second-order, as an average of second- and fourth-order approximations (the second-order term cannot cause cancellation since it only appears in the first summand).
A similar argument holds for every $a\in[0,1]$.
Similarly, it is true for the derivative in the $x_2$ direction, and the resulting gradient and divergence are also second-order accurate.

Next, we show that the grad-div is second-order accurate.
Using \eqref{eq:GradAndDiv},
\begin{equation*}
\grad_h^{2/3} (\div)_h = \begin{pmatrix}
(\partial_{x_1})_{h/2}^{2/3} (\partial_{x_1})_{h/2} & (\partial_{x_1})_{h/2}^{2/3} (\partial_{x_2})_{h/2} \\[0.5em]
(\partial_{x_2})_{h/2}^{2/3} (\partial_{x_1})_{h/2} & (\partial_{x_2})_{h/2}^{2/3} (\partial_{x_2})_{h/2}
\end{pmatrix}.
\end{equation*}
Since each of the entries of this matrix is a product of two approximations with second-order accuracy, the order is at least two.
To show that the order is at most two, we consider the resulting approximation of the second derivatives and cross derivatives.

First, $(\partial_{x_1})_{h/2}^{2/3} (\partial_{x_1})_{h/2}$ has second-order accuracy. 
Let $f(x_1,x_2)$ be a function.
Then a direct calculation shows
\begin{align*}
\textstyle
(\partial_{x_1})_{h/2}^{2/3} (\partial_{x_1})_{h/2}
&\textstyle = (\partial_{x_1})_{h/2}^{2/3}
\left(\frac{1}{h}\left(f\left(x_1+\frac{h}{2},x_2\right)-f\left(x_1-\frac{h}{2},x_2\right)\right)\right) \\
 & = \frac{1}{12h^2}
( f(x_1+h,x_2+h) - 2 f(x_1,x_2+h) + f(x_1-h,x_2+h) \\
 & \qquad \qquad  +
 10f(x_1+h,x_2) - 20 f(x_1,x_2) + 10f(x_1-h,x_2) \\
  & \qquad \qquad +
  f(x_1+h,x_2-h) - 2 f(x_1,x_2-h) + f(x_1-h,x_2-h) ).
\end{align*}
That is, in stencil notation
\begin{equation*}
\textstyle
(\partial_{x_1})_{h/2}^{2/3} (\partial_{x_1})_{h/2}
= \frac{1}{12h^2}\begin{bmatrix}
1 & -2 & 1 \\
10 & -20 & 10 \\
1 & -2 & 1
\end{bmatrix}
\end{equation*}
and this is only a second-order-accurate approximation of the continuous second derivative $\partial^2_{x_1,x_1}$, as can be seen from a direct Taylor expansion.
A similar calculation holds for $(\partial_{x_2})_{h/2}^{2/3} (\partial_{x_2})_{h/2}$.

Second, we show that the cross-derivative $(\partial_{x_1})_{h/2}^{2/3} (\partial_{x_2})_{h/2}$ is of second-order accuracy.
\begin{align*}
\textstyle
(\partial_{x_1})_{h/2}^{2/3} (\partial_{x_2})_{h/2}f &\textstyle = (\partial_{x_1})_{h/2}^{2/3} \left( \frac{1}{h} \left(f\left(x_1,x_2+\frac{h}{2}\right) - f\left(x_1,x_2-\frac{h}{2}\right)\right)\right) \\
\textstyle & \textstyle = \frac{1}{12h^2} \bigg(-f\left(x_1-\frac{h}{2},x_2+\frac{3h}{2}\right)+f\left(x_1+\frac{h}{2},x_2+\frac{3h}{2}\right) \\
\textstyle & \textstyle \qquad \qquad 
-9f\left(x_1-\frac{h}{2},x_2+\frac{h}{2}\right) + 9f\left(x_1+\frac{h}{2}, x_2+\frac{h}{2}\right)\\
\textstyle & \textstyle \qquad \qquad 
+9f\left(x_1-\frac{h}{2},x_2-\frac{h}{2}\right) -9f\left(x_1+\frac{h}{2}, x_2-\frac{h}{2}\right) \\
\textstyle & \textstyle \qquad \qquad 
+f\left(x_1-\frac{h}{2}, x_2-\frac{3h}{2}\right) -f\left(x_1+\frac{h}{2},x_2-\frac{3h}{2}\right) \bigg).
\end{align*}
In stencil notation,
\begin{equation*}
(\partial_{x_1})_{h/2}^{2/3} (\partial_{x_2})_{h/2}
=
\frac{1}{12h^2}\begin{bmatrix}
-1 & & 1 \\
 & & \\
 -9 & & 9 \\
  & * & \\
 9 & & -9 \\
  & & \\
1 & & -1
\end{bmatrix},
\end{equation*}
and again, a direct Taylor expansion shows that this is a second-order approximation of the continuous cross-derivative $\partial^2_{x_1,x_2}$.
A similar calculation holds for $(\partial_{x_2})_{h/2}^{2/3} (\partial_{x_1})_{h/2}$.

We conclude that the given approximation for the grad-div operator is second-order accurate for $a=2/3$. Hence, the overall discretization of \eqref{eq:elasticHelmEquivalent} is second-order accurate.
We note that this occurs even when the gradient and divergence are chosen to yield the fourth-order Laplacian when multiplied as div-grad.
Clearly, for other values of $a$, the order does not exceed $2$, for similar reasons.

\subsection{Numerical order of accuracy results}\label{subsec:OOA_numerical}

\begin{figure}
\begin{center}
	\newcommand{\image}[1]{\includegraphics[width=0.32\linewidth]{./#1}}
   \subfigure[\footnotesize Acoustic Helmholtz, 2D]{\image{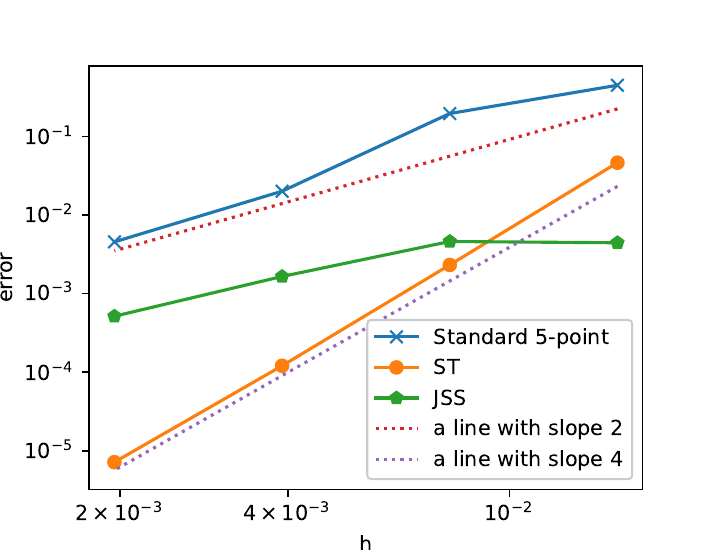}\label{fig:OOA_acoustic_mp2}}
       \subfigure[\footnotesize Elastic Helmholtz (mixed), 2D]{\image{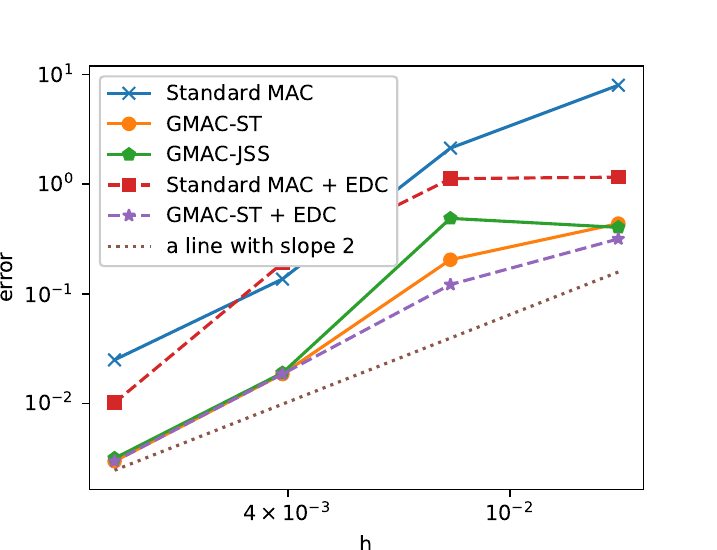}\label{fig:OOA_elastic_mp2}}
   \subfigure[\footnotesize Elastic Helmholtz, 3D]{\image{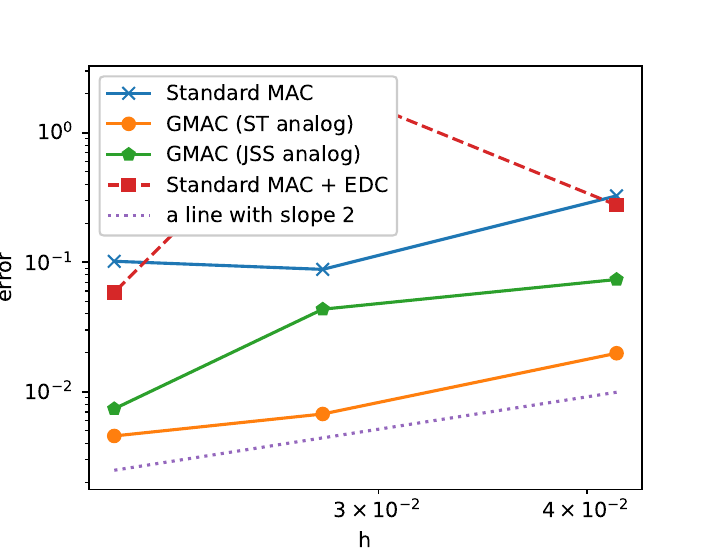}\label{fig:OOA3D_large}}
   \\
\end{center}
\caption{Truncation error vs. grid size in log-log plot for different discretizations of the acoustic and elastic Helmholtz equations.
On the left, the two-dimensional acoustic Helmholtz equation \eqref{eq:acousitcHelm} is discretized with standard second-order FD or with the fourth-order FD by Singer and Turkel \cite{singer1998high}, or by the second-order dispersion correction FD by Jo, Shin and Suh \cite{jo1996optimal}.
In the middle, the two-dimensional elastic Helmholtz equation in mixed formulation \eqref{eq:elasticMixed}
is discretized by standard MAC, GMAC-ST and GMAC-JSS.
On the right, the three-dimensional elastic Helmholtz equation \eqref{eq:elasticHelmEquivalent} is discretized by standard MAC, GMAC (ST analog), or GMAC (JSS analog).
Optimized constants for EDC are given in Table \ref{tab:tpts}.
}\label{fig:OOA}
\end{figure}

In this subsection, we give numerical experiments to test the performance of our discretization and dispersion correction method.
We performed the experiments in Julia \cite{Julia} using double-precision arithmetic.
We ran the experiments on a Lenovo ThinkPad X1 Carbon Gen 13 laptop running Windows 11, equipped with an Intel Core Ultra 7 265U processor and 64 GB of RAM.

To validate the calculation of the order of accuracy from Subsection \ref{subsec:OOA_analysis}, we show in Fig. \ref{fig:OOA} the truncation error in a log-log plot for different cases.
We calculate the truncation error for the elastic Helmholtz equation with a manufactured solution in two and three dimensions.
For comparison, we also give results for the acoustic Helmholtz equation \eqref{eq:acousitcHelm}.
In two dimensions, each of the equations is solved in the unit square $[0,1]^2$ on four grid sizes: $64^2$, $128^2$, $256^2$, and $512^2$ cells; in three dimensions, on three grid sizes: $24^3$, $36^3$, and $48^3$.

In Fig. \ref{fig:OOA_acoustic_mp2} we include results for the acoustic Helmholtz equation.
We discretize \eqref{eq:acousitcHelm} with a constant slowness model $\kappa=1$ and angular frequency $\omega = 32 \pi$ on a nodal grid with 4 grid points per shear wavelength on the smallest grid and 32 on the largest grid.
We use the oscillatory manufactured solution \begin{equation}\label{eq:manufactured_acoustic}
u_{exact}(x_1,x_2) = \sin(k_1 x_1)\cos(k_2 x_2),
\end{equation} 
where $k_1 = \omega$ and $k_2 = \frac{\omega}{2}$,
and impose Dirichlet boundary conditions with the values of $u_{exact}$ on the boundary.
For the right-hand side, we apply the continuous Helmholtz operator from \eqref{eq:acousitcHelm} to the manufactured solution, which gives
\begin{equation}\label{eq:RHS_acoustic}
-\Delta u_{exact} -\omega^2 u_{exact} = (k_1^2+k_2^2-\omega^2) u_{exact}.
\end{equation}
We discretize the right-hand side using the same stencil as for the mass matrix in \eqref{eq:JSS_sten} to maintain the order of accuracy; see \eqref{eq:harari}.
This is done by multiplying the vector on the right-hand side, obtained by sampling \eqref{eq:RHS_acoustic} at the grid nodes, by an appropriate mass matrix.
We present the normalized error 
\begin{equation}
\bfe = \sqrt{h_1 h_2}\|\bfu_{num}-\bfu_{exact}\|_2
\end{equation}
in a log-log plot, where $\bfu_{exact}$ is a vector of samples of $u_{exact}$ in the grid locations and $\bfu_{num}$ is the direct solution of the algebraic equation.
The results in Fig. \ref{fig:OOA_acoustic_mp2} show second-order accuracy for the standard 5-point stencil and for the JSS stencil presented by \cite{jo1996optimal}.
For the ST stencil presented by Singer and Turkel in \cite{singer1998high}, we demonstrate fourth-order accuracy, as expected from their theoretical calculations.

In fig. \ref{fig:OOA_elastic_mp2} we present results for the two-dimensional elastic Helmholtz equation in mixed formulation.
We discretize 
 \eqref{eq:elasticMixed} 
in homogeneous media, with constant Lam{\'e} coefficients $\lambda=\mu=1$ corresponding to Poisson ratio $\sigma=0.25$, and with constant density $\rho=1$, using the following discretizations:
standard MAC, GMAC-ST or GMAC-JSS (see Subsection \ref{subsec:numerical2D}). 
We use the vector-valued manufactured solution
\begin{equation}\label{eq:manufactured}
\vec{u}_{exact}(x_1,x_2)
=\vec{u}_s(x_1,x_2)+\vec{u}_p(x_1,x_2)=\begin{pmatrix}\sin (k_1 x_1)
\cos(k_2 x_2)\\
 -
\frac{k_1}{k_2} \cos(k_1 x_1)\sin(k_2 x_2)\end{pmatrix} +
\begin{pmatrix}\sin(k_p x_1)\cos(k_p x_2)\\ \cos(k_p x_1)\sin(k_p x_2)\end{pmatrix},
\end{equation}
where $k_1=\frac{\omega}{\sqrt{\mu}}$, $k_2=\frac{\omega}{2\sqrt{\mu}}$ and $k_p = \frac{\omega}{\sqrt{\lambda+2\mu}}$. 
This is analogous to the oscillatory solution in the acoustic case, where $\vec{u}_s$ is divergence-free and represents shear waves only.
\footnote{We avoid choosing $k_1=k_2$ because it causes the leading error term to cancel by symmetry, and the GMAC-ST curve appears to have fourth-order accuracy, whereas the discretization is generally only second-order accurate.}.
We impose non-homogeneous Dirichlet boundary conditions with the value of the $\vec u_{exact}$ on the boundaries. 
The right-hand side is calculated by applying the continuous elastic Helmholtz operator in the mixed formulation \eqref{eq:elasticMixed} to the manufactured solution.
The resulting forcing term is therefore 
\begin{equation}
\vec{q}_s(x_1,x_2,x_3) = 
\left(\mu (k_1^2+k_2^2)-\omega^2\rho\right)\vec {u}_s + \left(2(\lambda+2\mu)k_p^2-\omega^2\rho\right)\vec{u}_p.
\end{equation}
We discretize this right-hand side using the same mass matrix we used to spread the mass term in the leading block (see Subsection \ref{subsec:schemes2D}).
Note that this right-hand side lives in the faces space in MAC locations, so each of the blocks in this vector is multiplied by an adjusted mass matrix for the corresponding face.
We concatenate the resulting vector with a vector of zeros of the size of the $p$-space, to fit the right-hand-side for the mixed formulation \eqref{eq:elasticMixed}.
Similar to the acoustic case, we use angular frequency $\omega = 32 \pi$ that corresponds to $G_s=4$ on the smallest grid and $G_s=32$ on the largest grid.


Fig.  \ref{fig:OOA_elastic_mp2} shows second-order accuracy for the standard MAC, GMAC-ST and GMAC-JSS, as expected from Subsection \ref{subsec:OOA_analysis}. We also compare the truncation error with and without EDC (see Section \ref{sec:lame}), for standard MAC and for GMAC-ST.
We first optimize $t_p$ and $t_s$ for \eqref{eq:s_def} using the theory described in Section \ref{sec:analysis}. 
In two dimensions,
we sample the dispersion relation at angles in $[0,45]$ degrees, with 5-degree jumps between samples.
In three dimensions, we sample in 15-degree jumps the directions $\phi=[\phi_{azimuthal},\phi_{polar}]$ for which 
$\phi_{azimuthal}\in\left[0,\frac{\pi}{4}\right]$ and 
$\phi_{polar}\in\left[\frac{\pi}{2}-\arccos\left(\frac{1}{\sqrt{3}}\right),\frac{\pi}{2}\right]$.
For each direction $\phi$, we sample the dispersion error \eqref{eq:disp_err} at a ray angle $\phi$ in jumps of $0.0001$.
Then, we sum the absolute values of the shear and pressure dispersion errors, and determine $t_p$ and $t_s$ that minimize the summed error over propagation directions.
The resulting optimal values are reported in Table \ref{tab:tpts}.
As seen in Fig. \ref{fig:OOA_elastic_mp2},
the EDC process reduces the error for both examined discretizations.
The reduction is more significant for the standard MAC, since the GMAC-ST already has low dispersion before the additional correction.

We perform a similar experiment for the three-dimensional elastic Helmholtz equation in grad-div formulation \eqref{eq:elasticHelmEquivalent}, which is the Schur complement of the mixed formulation \eqref{eq:elasticMixed}.
The grad-div formulation is used here for memory considerations, and is discretized by the same stencils for gradient, divergence, Laplacian, and mass described in Subsection \ref{subsec:schemes2D}.
That is, the GMAC discretization in this case is the Schur complement of \eqref{eq:MAC-dispersion}.
We solve the three-dimensional \eqref{eq:elasticHelmEquivalent} with constant Lam{\'e} coefficients $\lambda=\mu=1$ corresponding to Poisson ratio $\sigma=0.25$, and with constant density $\rho=1$.
We use the manufactured solution
\begin{equation}\label{eq:manufactured3D}
\vec{u}_{exact}(x_1,x_2,x_3)=\vec{u}_s(x_1,x_2,x_3)+\vec{u}_p(x_1,x_2,x_3)
\end{equation}
where
\begin{equation}\label{eq:usup3d}
\vec{u}_s=
\frac{1}{k_s}
\begin{pmatrix}
\sin(k_s x_1) (\cos(k_s x_2)\sin(k_s x_3) - \sin(k_s x_2)\cos(k_s x_3)) \\
\sin(k_s x_2) (\cos(k_s x_3)\sin(k_s x_1) - \sin(k_s x_3)\cos(k_s x_1)) \\
\sin(k_s x_3) (\cos(k_s x_1)\sin(k_s x_2) - \sin(k_s x_1)\cos(k_s x_2))
\end{pmatrix}, \, 
\vec{u}_p = \begin{pmatrix}
\sin(k_p x_1)\cos(k_p x_2)\cos(k_p x_3) \\
\cos(k_p x_1)\sin(k_p x_2)\cos(k_p x_3) \\
\cos(k_p x_1)\cos(k_p x_2)\sin(k_p x_3)
\end{pmatrix}
\end{equation}
with $k_s = \frac{\omega}{\sqrt{\mu}}$ and $k_p = \frac{\omega}{\sqrt{\lambda+2\mu}}$.
The forcing term
is therefore
\begin{equation} \label{eq:forcing3D}
\vec{q}_s(x_1,x_2,x_3) = \frac{1}{k_s}
\left(3\mu k_s^2 -\omega^2\rho\right)\vec {u}_s + \left(3(\lambda+2\mu)k_p^2-\omega^2\rho\right)\vec{u}_p.
\end{equation}
We use an angular frequency that corresponds to $G_s=4$ on the smallest grid and $G_s=8$ on the largest grid.

The results in Fig. \ref{fig:OOA3D_large} demonstrate that the standard MAC, the GMAC (ST analog), and the GMAC (JSS analog) are second-order accurate, as expected.
Although they have the same order of accuracy, the GMAC (ST analog) has the lowest truncation error.
Adding EDC is efficient only for standard MAC.

\begin{table}
\centering
\begin{tabular}{l|cccc|ccc}
  \toprule
  \mc{8}{c}{Optimized $t_p$ and $t_s$ for EDC} \\
  \midrule
&  \mc{4}{c|}{2D Standard MAC} &  \mc{3}{c}{3D Standard MAC}  \\
     \midrule
  & $G_s = 4$ & $G_s = 8$ & $G_s = 16$ & $G_s = 32$  & $G_s = 4$ & $G_s = 6$ & $G_s = 8$ \\
    \midrule
   $t_p$ & 0.971 & 0.993 & 0.998 & 1.0 & 0.982 & 0.992 & 0.996 \\
   $t_s$ & 0.913 & 0.978 & 0.996 & 0.999 & 0.948 & 0.977 & 0.987 \\
     \midrule
&  \mc{4}{c|}{2D GMAC-ST}
&  \mc{3}{c}{3D GMAC (ST analog)} \\
     \midrule
  & $G_s = 4$ & $G_s = 8$ & $G_s = 16$ & $G_s = 32$  & $G_s = 4$ & $G_s = 6$ & $G_s = 8$ \\
       \midrule
   $t_p$ & 0.999 & 0.9999 & 1 & 1 & 1 & 1 & 1 \\
   $t_s$ & 0.994 & 0.9998 & 1 & 1 & 1 & 1 & 1 \\
 \bottomrule
 \end{tabular}
\caption{Optimal $t_p$ and $t_s$ for \eqref{eq:s_def}, minimizing the sum of absolute values of dispersion errors \eqref{eq:disp_err} over different propagation directions, for Poisson ratio $\sigma=0.25$.
}
\label{tab:tpts}
\end{table}

\paragraph{Summary of recommended choices}
As with acoustic dispersion correction methods, the same holds for their elastic analogs that we developed in this work.
In two dimensions, every elastic analog of one of the compact acoustic methods considered in Section~\ref{sec:acoustic} is physically reliable. 
While the recommendation depends on the aim, GMAC-ST and GMAC-JSS give similar results, with the former slightly outperforming the latter in terms of polarization and the latter slightly outperforming the former in terms of dispersion.
One can improve the dispersion correction by optimizing the GMAC weights for each $G_s$ value separately.
However, in three dimensions, polarization imposes an additional constraint. 
We therefore recommend the GMAC (ST analog).
The combination with EDC is mostly effective where the dispersion is significant, i.e., for standard MAC or for extremely small values of $G_s$.

\section{Computational results for heterogeneous media}\label{sec:heterogeneous}

\begin{figure}
\begin{center}
	\newcommand{\image}[1]{\includegraphics[width=0.48\linewidth]{./#1}}
    \subfigure[\footnotesize Linear media, $G_s \geq 5$, 0$^\circ$ section]{\image{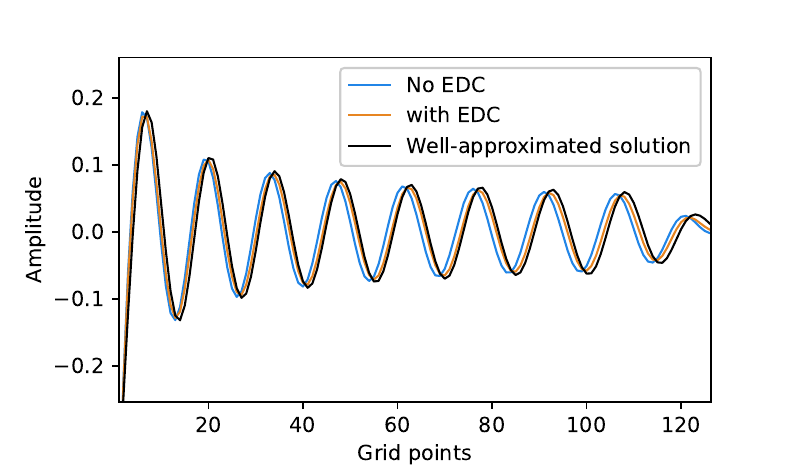}\label{fig:dispersionPlot0}}
   \subfigure[\footnotesize Linear media, $G_s \geq 5$, 45$^\circ$ section]{\image{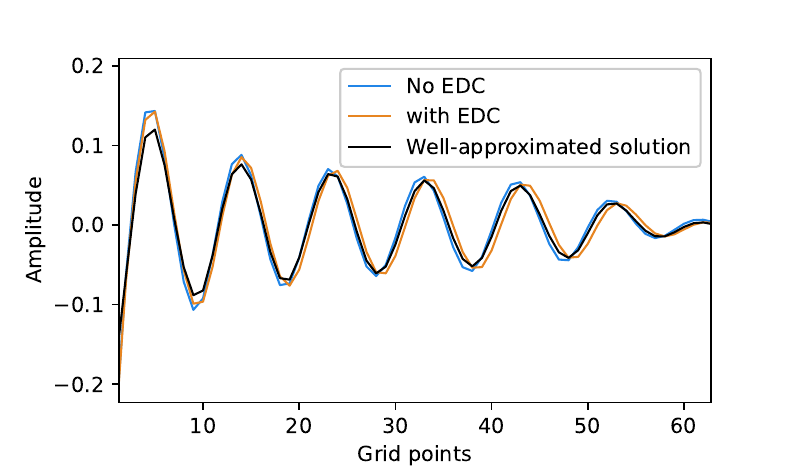}\label{fig:dispersionPlot45}}\\
\end{center}
\caption{A section of the pressure component of the solution of the elastic Helmholtz equation in mixed formulation in linear media, discretized by the standard MAC discretization on the unit square with a grid of $128\times128$ cells, with or without EDC from Section \ref{sec:lame}, compared to a reference solution achieved by approximating the same equation on a denser grid of $512\times512$ cells. The theoretically calculated shifts used in \eqref{eq:s_def} for the elastic dispersion correction are $t_s \in [0.95,0.997]$ and $t_p \in [0.99,1]$.
}\label{fig:dispersionPlot}
\end{figure}

In this section, we provide numerical evidence for the efficiency of our EDC process described in Section \ref{sec:lame} for the linearly varying Lam{\'e} coefficients $\lambda\in[4,10],$ $\mu\in[1,5]$ (corresponding to a Poisson ratio range $\sigma\in[0.357,4]$) and density $\rho\in[1,2]$. 
For simplicity, we use standard MAC, where the EDC effect is more pronounced.
We calculate optimized values for $t_p$ and $t_s$ in \eqref{eq:s_def} by sampling the dispersion relation described in Section \ref{sec:analysis}.
We calculate the dispersion error \eqref{eq:disp_err} for angles varying from $0^\circ$ to $45^\circ$ with jumps of $5^\circ$, by sampling the symbolic eigenvalues of the shear and pressure modes over a ray of the propagation direction with jumps of 0.001. We minimize the sum of the absolute errors across all propagation directions.
We repeat the optimization process separately for the values 
$\lambda$, $ \mu$, and $\rho$ at the bottom and top of the domain and use the resulting values to construct linearly varying $t_p$ and $t_s$. 
For the first step of the EDC process, we set 
\begin{equation}
\textstyle
\omega^* = \left(\frac{\min t_s + \max t_s}{2}\right)\omega.
\end{equation}
For the second step of the process, we compute a spatially variable $s$ from \eqref{eq:s_def} using the linearly varying $t_p$ and $t_s$ by point-wise operations.
For the last step, we apply the point-wise multiplication             $\lambda^* \!=\!  s \odot \lambda$.

In Fig. \ref{fig:dispersionPlot}, we solve the two-dimensional elastic Helmholtz equation in a mixed formulation
on the unit square $[0,1]^2$.
The right-hand side is a point source located in the center of the upper boundary of the domain.
We discretize equation \eqref{eq:elasticMixed} using the standard MAC, \eqref{eq:MAC-dispersion}, with weights $[a,b,c] = [1,1,0]$.
We use an absorbing boundary layer \cite{engquist1977absorbing} to prevent reflections from the boundaries. 
We implement the layer by adding a gradually increasing complex shift toward the boundaries (except at the edge where the source is located; there, we apply Neumann boundary conditions) and set its width to $\frac{1}{5}$ of the total domain width.
We display a section of the solution in different propagation directions and compare it with a well-approximated solution on a denser grid, with and without applying EDC.
The original $128\times128$ grid satisfies $G_s=5$, and the dense $512\times512$ grid satisfies $G_s=20$ in the most challenging regime of the domain
We show sections of the solution's pressure component at $0^\circ$ and $45^\circ$ relative to the vertical.
The $0^\circ$ section, shown in Fig. \ref{fig:dispersionPlot0}, extends from the source to the middle of the bottom boundary of the domain.
The $45^\circ$ section, shown in Fig. \ref{fig:dispersionPlot45}, goes from the source to the middle of a side boundary and therefore contains only half the grid points.
Since the propagation of a discretely approximated wave is direction-dependent, phase alignment cannot be perfect for all propagation directions.
However, we observe a better phase agreement after applying EDC.

Although analyzing heterogeneous elastic dispersion correction is beyond the scope of this paper, we provide numerical evidence of its efficiency.
Building a theory to analyze it is particularly challenging for nonsmooth coefficients.
The numerical results above demonstrate that our EDC, with nonconstant shifting parameters $t_p$ and $t_s$ in \eqref{eq:s_def} that act pointwise, improves phase alignment in mildly heterogeneous media with smoothly varying coefficients.

\section{Conclusion}\label{sec:conclusion}
We have introduced a family of GMAC (generalized marker and cell) discretizations for the elastic Helmholtz equation in mixed formulation. 
Our discretization scheme is compact and enables significant dispersion correction by choosing appropriate stencil weights.
Moreover, we present the EDC (elastic dispersion correction) process, which further corrects the dispersion by shifting the shear and pressure wavenumbers simultaneously.
Altogether, our method extends dispersion correction methods originally developed for the acoustic Helmholtz equation to the elastic Helmholtz equation. 
Although the GMAC discretization scheme is generally non-symmetric, it is free of numerical dissipation and shear-mode separation, as we prove and demonstrate numerically. 
One limitation of our method is that the high order accuracy achieved in an acoustic discretization may not carry over to its elastic counterpart. 
Nevertheless, the dispersion-correction properties persist even in low-order discretizations. 
Another limitation is that not every analog of a physical discretization of the acoustic Helmholtz equation yields a physical discretization of the elastic Helmholtz equation; 
for certain weight choices, the polarization error becomes too large. 
We compare the dispersion of our discretizations with that of existing compact discretizations for elastic waves, showing that our method yields about $\times 7$ lower dispersion error.
Moreover, the dispersion properties of our discretization make it a good alternative even for non-compact discretization for elastic wave propagation.
Overall, our framework enables a variety of dispersion-correction methods for the elastic Helmholtz equation using compact discretizations, thereby reducing dispersion, polarization, and truncation errors.
The analysis and optimization presented are restricted to the constant-coefficient elastic Helmholtz equation. 
We have shown numerical evidence for the efficiency of our elastic dispersion correction in heterogeneous media.

\section*{CRediT author contribution}

Rachel Yovel: conceptualization, methodology, formal analysis, software, validation, writing --- original draft, writing --- review \& editing.

Eli Turkel: methodology, supervision, writing --- review and editing.

Eran Treister: methodology, supervision, software, writing --- review and editing.

\bibliographystyle{elsarticle-num}


\bibliography{main.bbl}

\end{document}